\documentclass[a4paper,11pt]{article}
\usepackage{amsmath,amsthm,amssymb}
\usepackage{enumerate}

\usepackage[english]{babel}
\usepackage[utf8]{inputenc}
\usepackage[a4paper]{geometry}
\usepackage{latexsym}
\usepackage{amscd}
\usepackage{graphics}
\usepackage{color}
\usepackage{array}
\usepackage{diagrams}
\usepackage{mathrsfs}
\usepackage{graphicx}
\usepackage{stmaryrd}
\usepackage{mathabx}
\usepackage{dsfont}
\usepackage{comment}
\usepackage{tikz}
\usepackage{tikz-cd}
\usepackage{calligra}

\def\subjclass#1{{\renewcommand{\thefootnote}{}%
\footnote{\emph{Mathematics Subject Classification (2010):} #1}}}
\def\keywords#1{{\renewcommand{\thefootnote}{}%
\footnote{\emph{Keywords:} #1}}}

\newtheorem{thm}{Theorem}[section]
\newtheorem{cor}[thm]{Corollary}
\newtheorem{lem}[thm]{Lemma}

\newtheorem{prop}[thm]{Proposition}
\newtheorem{conj}[thm]{Conjecture}
\newtheorem{fact}[thm]{Fact}
\newtheorem{question}[thm]{Question}

\theoremstyle{definition}
\newtheorem{defin}[thm]{Definition}
\newtheorem{rem}[thm]{Remark}

\numberwithin{equation}{section}

\newcommand{\R}{\mathbb{R}}
\newcommand{\Ros}{\mathcal{R}}
\newcommand{\X}{\mathcal{X}}
\newcommand{\elinf}{[\omega]^\omega}
\newcommand{\Sil}{\mathcal{N}}
\newcommand{\St}{\mathcal{S}}
\newcommand{\G}{\mathcal{G}}
\newcommand{\I}{\textbf{I}}
\newcommand{\II}{\textbf{II}}
\newcommand{\A}{\mathcal{A}}
\newcommand{\K}{\mathcal{K}}

\DeclareMathOperator{\supp}{supp}
\DeclareMathOperator{\spa}{span}
\DeclareMathOperator{\Seq}{Seq}
\DeclareMathOperator{\bs}{bs}

\newcommand{\st}{\text{{\calligra \LARGE s}}\,}
\newcommand{\stp}{\text{{\calligra \large s}}\,}
\newcommand{\Ast}{A_{\text{{\calligra \Large s}}\,}}
\newcommand{\Asti}[1]{A_{{\text{{\calligra \Large s}}\,}_{#1}}}

\newcommand{\concat}{%
  \mathbin{\raisebox{1ex}{\scalebox{.7}{$\frown$}}}%
}

\makeatletter
\newcommand{\alaligne}{~\vspace*{\topsep}\nobreak\@afterheading}
\makeatother

\begin{document}






\title{Ramsey theory without pigeonhole principle and the adversarial Ramsey principle}

\author{N. de Rancourt}

\date{}

\maketitle




\begin{abstract}
We develop a general framework for infinite-dimensional Ramsey theory with and without pigeonhole principle, inspired by Gowers' Ramsey-type theorem for block sequences in Banach spaces and by its exact version proved by Rosendal. In this framework, we prove the adversarial Ramsey principle for Borel sets, a result conjectured by Rosendal that generalizes at the same time his version of Gowers' theorem and Borel determinacy of games on integers.

\end{abstract}

\section{Introduction}\label{par1}

This paper has two main goals. The first is to develop an abstract formalism for infinite-dimensional Ramsey theory, enabling to prove both Ramsey results with a pigeonhole principle (like Mathias--Silver's theorem \cite{mathias, silver}) and Ramsey results without a pigeonhole principle, like Gowers' Ramsey-type theorem for block sequences in Banach spaces \cite{GowersRamsey} and its exact version given by Rosendal \cite{RosendalGowers}. The second goal is to prove the \emph{adversarial Ramsey principle} for Borel sets, a result conjectured by Rosendal \cite{Rosendaladverse}, unifying his exact version of Gowers' theorem with Borel determinacy of games on integers. In order to motivate the rest of the paper and to recall the statement of the latter results and conjectures, let us begin with some history.

\smallskip

Infinite-dimensional Ramsey theory is a branch of Ramsey theory where we color \emph{infinite-dimensional} objects, that are, \emph{sequences of points} of some space, and where we want to find homogeneous subspaces. The fundamental result in infinite-dimensional Ramsey theory is Mathias--Silver's theorem, proved independently in 1968 by Mathias \cite{mathias} and in 1970 by Silver \cite{silver}, saying that if $\X$ is an analytic subset of $\elinf$ (the set of all infinite subsets of $\omega$, endowed with the topology inherited from the Cantor space $\mathcal{P}(\omega) = 2^\omega$ with the product topology) then for every infinite $M \subseteq \omega$, there exists an infinite $N \subseteq M$ (called a \emph{homogeneous} set) such that either for every infinite $S \subseteq N$, we have $S \in \X$, or for every infinite $S \subseteq N$, we have $S \in \X^c$. (A set $\X \subseteq \elinf$ satisfying the conclusion of this theorem will be called a \emph{Ramsey set}.) An important remark is that in this theorem (and in particular, in all of its proofs using the classical technique of combinatorial forcing, see for example \cite{todorcevicorange}, Chapter 1, for a presentation of this technique), the sets $M$ and $N$ are often seen as \emph{subspaces} (elements of a poset) while the set $S$ is rather seen as an \emph{infinite sequence}, the increasing sequence of its elements; this distinction between subspaces and sequences of points will appear in all results presented in this paper.

\smallskip

The proof of Mathias--Silver's theorem uses in an essential way the \emph{pigeonhole principle}, i.e. the trivial fact that for every infinite $M \subseteq \omega$ and every $A \subseteq \omega$, there exists an infinite $N \subseteq M$ such that either $N \subseteq A$, or $N \subseteq A^c$. In the decades that followed the proof of this theorem, several similar results arose in different contexts (words, trees, etc.). All of these results have the same form: we color infinite sequences of points satisfying some structural condition (being increasing, being block sequences, etc.) and the theorem ensures that we can find a monochromatic subspace. The proof of each of these results relies on an analogue of the pigeonhole principle, whose proof is in most cases way less trivial as in $\omega$. In general, a \emph{pigeonhole principle} is a one-dimensional Ramsey result, i.e. a result where we color points and we want to find a monochromatic subspace. A lot of these pigeonhole principles, and the infinite-dimensional Ramsey results they imply, can be found in Todor\v{c}evi\'c's book \cite{todorcevicorange}, where a general framework to deduce an infinite dimensional Ramsey result from its associated pigeonhole principle is also developped.

\smallskip

The first infinite-dimensional Ramsey-type result that was not relying on a pigeonhole principle was proved by Gowers, in the 90's. The aim of Gowers was to solve a celebrated problem asked by Banach, the homogeneous space problem, asking whether $\ell_2$ was the only infinite-dimensional Banach space, up to isomorphism, that was isomorphic to all of its closed, infinite-dimensional subspaces. Gowers proved a dichotomy \cite{GowersRamsey} that, combined with a result by Komorowski and Tomczak-Jaegermann \cite{KomorowskiTomczakJaegermann}, provided a positive answer to Banach's question. The proof of this dichotomy relies on a Ramsey-type theorem in separable Banach spaces, that we will state now. The reader who is not familiar with Banach space geometry can skip this part, since it will only be relevant to understanding sections \ref{par4} and \ref{par5} of this paper.

\smallskip

In this paper, to save writing, we will only consider real Banach spaces, but the results we present here adapt to the complex case. Let $E$ be a Banach space. Recall that a \emph{(Schauder) basis} of $E$ is a sequence $(e_i)_{i \in \omega}$ such that every $x \in E$ can be written in a unique way as an infinite sum $\sum_{i = 0}^\infty x^i e_i$, where $x^i \in \R$. In this case, the \emph{support} of the vector $x$, denoted by $\supp(x)$, is defined as the set $\{i \in \omega \mid x^i \neq 0\}$. A \emph{block sequence} of $(e_i)$ is an infinite sequence $(x_n)_{n\in \omega}$ of nonzero vectors of $E$ with $\supp(x_0) < \supp(x_1) < \ldots$ (here, for two nonempty sets of integers $A$ and $B$, the notation $A < B$ means that $\forall i \in A \; \forall j \in B \; i < j$). It can be shown that a block sequence is a basis of the closed subspace it spans, such a space being called a \emph{block subspace}.  A basis, or a block sequence, is said to be \emph{normalized} if all of its terms have norm $1$. In the rest of this article, unless otherwise specified, every basis and every block sequence in a Banach space will be normalized.

\smallskip

For $X$ a block subspace of $E$, let $[X]$ denote the set of block sequences all of whose terms are in $X$. We can equip $[E]$ with a natural topology by seeing it as a subspace of $(S_E)^\omega$ with the product topology (where $S_E$ denotes the unit sphere of $E$ with the norm topology), which makes it a Polish space. For $\X \subseteq [E]$ and $\Delta = (\Delta_n)_{n \in \omega}$ a sequence of positive real numbers, let $(\X)_\Delta$ denote the \emph{$\Delta$-expansion} of $\X$, that is, the set $\{(x_n)_{n \in \omega} \in [E]\, | \,\exists (y_n)_{n \in \omega} \in \X \; \forall n \in \omega \; \|x_n - y_n\| \leqslant \Delta_n\}$. In order to state Gowers' theorem, we need a last definition.

\begin{defin}

Let $X$ be a block subspace of $E$. \emph{Gowers' game} below $X$, denoted by $G_X$, is the following infinite two-players game (whose players will be denoted by \I{} and \II{}):

\smallskip

\begin{tabular}{ccccccc}
\textbf{I} & $Y_0$ & & $Y_1$ & & $\hdots$ & \\
\textbf{II} & & $y_0$ & & $y_1$ & & $\hdots$ 
\end{tabular}

\smallskip

\noindent where the $Y_i$'s are block subspaces of $X$, and the $y_i$'s are normalized vectors of $E$ with finite support, with the constraints for \II{} that for all $i \in \omega$, $y_i \in Y_i$ and $\supp(y_i) < \supp(y_{i + 1})$. The outcome of the game is the sequence $(y_i)_{i \in \omega} \in [E]$.

\end{defin}

In this paper, when dealing with games, we shall use a convention introduced by Rosendal: we omit to set a winning condition when defining a game, but rather associate an \emph{outcome} to the game and say that a player has a strategy to force this outcome to belong to some fixed set. For example, saying that player \II{} has a strategy to reach a set $\X \subseteq [E]$ in the game $G_X$ means that she has a winning strategy in the game whose rules are those of $G_X$ and whose winning condition is the fact that the outcome belongs to $\X$.

\smallskip

We can now state Gowers' theorem:

\begin{thm}[Gowers' Ramsey-type theorem]\label{GowersThm}

Let $\X \subseteq [E]$ be an analytic set, $X \subseteq E$ be a block subspace, and $\Delta$ be an infinite sequence of positive real numbers. Then there exists a block subspace $Y$ of $X$ such that either $[Y] \subseteq \X^c$, or player \II{} has a strategy in $G_Y$ to reach $(\X)_\Delta$.

\end{thm}

While one of the possible conclusions of this theorem, $[X] \subseteq \X^c$, is very similar to ``For every infinite $S \subseteq M$, we have $S \in \X^c$'' in Mathias--Silver's theorem, the other one is much weaker, for to reasons: the use of metrical approximation and the use of a game. As we will see later, the necessity of the approximation is due to a lack of finiteness, while the necessity for one of the possible conclusions to involve a game matters much more and is due to the lack of a pigeonhole principle in this context. In some Banach spaces, a pigeonhole principle holds, and in these spaces, Gowers gave a strengthening of his theorem, involving no game, that we will introduce now. We start by stating the general form of the relevant pigeonhole principle in the context of Banach spaces; since an exact pigeonhole principle is never satisfied in this context, and would anyways be useless since approximation is needed for other reasons, we will only state an approximate pigeonhole principle. For a Banach space $E$, a set $A \subseteq S_E$, and $\delta > 0$, denote by $(A)_\delta$ the \textit{$\delta$-expansion} of $A$, that is, the set $\{x \in S_E \, \mid \, \exists y \in A \; \|x - y\| \leqslant \delta\}$.

\begin{defin}

Say that a Banach space $E$ with a Schauder basis satisfies the \emph{approximate pigeonhole principle} if for every $A \subseteq S_E$, for every block subspace $X \subseteq E$, and for every $\delta > 0$, there exists a block subspace $Y \subseteq X$ such that either $S_Y \subseteq A^c$, or $S_Y \subseteq (A)_\delta$.

\end{defin}

Recall that an infinite-dimensional Banach space $E$ is said to be \emph{$c_0$-saturated} if $c_0$ can be embedded in all of its infinite-dimensional, closed subspaces. A combination of results by Milman \cite{milman}, Gowers \cite{GowersLipschitz}, and Odell and Schlumprecht \cite{OdellSchlumprecht}, shows the following:

\begin{thm}\label{pigeonc0}

A space $E$ with a Schauder basis satisfies the approximate pigeonhole principle if an only if it is $c_0$-saturated.

\end{thm}

Thus, in $c_0$-saturated spaces, we have a strengthening of Gowers' theorem:

\begin{thm}[Gowers' Ramsey-type theorem for $c_0$]\label{Gowersc0theorem}

Suppose that $E$ is $c_0$-saturated. Let $\X \subseteq [E]$ be an analytic set, $X\subseteq E$ be a block subspace, and $\Delta$ be an infinite sequence of positive real numbers. Then there exists a block subspace $Y$ of $X$ such that either $[Y] \subseteq \X^c$, or $[Y] \subseteq (\X)_\Delta$.

\end{thm}

For a complete survey of Gowers' Ramsey-type theory in Banach spaces, see \cite{todorcevicbleu}, Part B, Chapter IV.

\smallskip

In 2010, in \cite{RosendalGowers}, Rosendal proved an exact version (without approximation) of Gowers' theorem, in countable vector spaces, which easily implies Gowers' theorem in Banach spaces. In this theorem, in order to remove the approximation, the non-game-theoretical conclusion has to be weakened by introducing a new game, the \emph{asymptotic game}. We present here Rosendal's theorem in more details. Let $E$ be a countably infinite-dimensional vector space over an at most countable field $K$ and $(e_i)_{i \in \omega}$ be a basis (in the algebraic sense) of $E$. The notions of \emph{support} of a vector and of a \emph{block sequence} relative to this basis are defined exactly in the same way as in a Banach space with a Schauder basis (this time, since there is no norm, there is obviously no notion of normalized bases or block sequences). A \emph{block subspace} is a subspace of $E$ spanned by a block sequence. Observe that in this setting, since every vector has finite support, every infinite-dimensional subspace of $E$ has a further subspace that is a block subspace. Given a block subspace $X \subseteq E$, define the two following games:

\begin{defin}

\alaligne

\begin{enumerate}

\item \emph{Gowers' game} below $X$, denoted by $G_X$, is defined in the following way:

\smallskip

\begin{tabular}{ccccccc}
\textbf{I} & $Y_0$ & & $Y_1$ & & $\hdots$ & \\
\textbf{II} & & $y_0$ & & $y_1$ & & $\hdots$ 
\end{tabular}

\smallskip

\noindent where the $Y_i$'s are block subspaces of $X$, and the $y_i$'s are nonzero elements of $E$, with the constraint for \II{} that for all $i \in \omega$, $y_i \in Y_i$. The outcome of the game is the sequence $(y_i)_{i \in \omega} \in E^\omega$.

\item The \emph{asymptotic game} below $X$, denoted by $F_X$, is defined in the same way as $G_X$, except that this time, the $Y_i$'s are moreover required to have finite codimension in $X$.

\end{enumerate}

\end{defin}

We endow $E$ with the discrete topology and $E^\omega$ with the product topology; since $E$ is countabe, $E^\omega$ is a Polish space. Rosendal's theorem is the following:

\begin{thm}[Rosendal]\label{RosendalGowersthm}

Let $\X$ be an analytic subset of $E^\omega$. Then for every block subspace $X \subseteq E$, there exists a block subspace $Y \subseteq X$ such that either \I{} has a strategy in $F_Y$ to reach $\X^c$, or \II{} has a strategy in $G_Y$ to reach $\X$.

\end{thm}

Say that a set $\X \subseteq E^\omega$ is \emph{strategically Ramsey} if it satisfies the conclusion of this theorem. Here, the use of an asymptotic game in one side of the alternative is (as we will see in the rest of this paper) not much weaker than a non-game-theoretical conclusion as in Mathias--Silver's theorem.

\smallskip

In the same paper as the last theorem, Rosendal, inspired by the work of Pelczar \cite{Pelczar}, and by a common work with Ferenczi \cite{fero}, introduced a new Ramsey principle
which is, unlike Theorem \ref{RosendalGowersthm}, symmetrical. His result was then refined in \cite{Rosendaladverse}. It involves two games, known as the \emph{adversarial Gowers' games}, obtained by mixing the games $G_X$ and $F_X$.

\begin{defin}

\alaligne

\begin{enumerate}

\item For a block subspace $X \subseteq E$, the game $A_X$ is defined in the following way:

\smallskip

\begin{tabular}{cccccccc}
\textbf{I} & & $x_0, \, Y_0$ & & $x_1, \, Y_1$ & & $\hdots$ \\
\textbf{II} & $X_0$ & & $y_0, \, X_1$ & & $y_1, \, X_2 $ & & $\hdots$ 
\end{tabular}

\smallskip

\noindent where the $x_i$'s and the $y_i$'s are nonzero vectors of $X$, the $X_i$'s are block subspaces of $X$, and the $Y_i$'s are block subspaces of $X$ with finite codimension. The rules are the following:

\begin{itemize}

\item for \textbf{I}: for all $i \in \omega$, $x_i \in X_i$;

\item for \textbf{II}: for all $i \in \omega$, $y_i \in Y_i$;

\end{itemize}

\noindent and the outcome of the game is the sequence $(x_0, y_0, x_1, y_1, \ldots) \in E^\omega$.

\item The game $B_X$ is defined in the same way as $A_X$, except that this time the $X_i$'s are required to have finite codimension in $X$, whereas the $Y_i$'s can be arbitrary block subspaces of $X$.

\end{enumerate}

\end{defin}

The result Rosendal proves in \cite{Rosendaladverse} is the following:

\begin{thm}[Rosendal]\label{Adversarialconcrete}

Let $\X \subseteq E^\omega$ be $\boldsymbol{\Sigma}_3^0$ or $\boldsymbol{\Pi}_3^0$. Then for every block subspace $X \subseteq E$, there exists a block subspace $Y \subseteq X$ such that either \I{} has a strategy in $A_Y$ to reach $\X$, or \II{} has a strategy in $B_Y$ to reach $\X^c$.

\end{thm}

Say that a set $\X \subseteq E^\omega$ is \emph{adversarially Ramsey} if it satisfies the conclusion of this theorem. Then, a natural question to ask is for which complexity of the set $\X$ one can ensure that it is adversarially Ramsey.

\smallskip

There are two things to observe. First, let $\X \subseteq E^\omega$ and define $\X' = \{(x_i)_{i \in \omega} \in E^\omega \mid (x_{2i})_{i \in \omega} \in \X\}$. Then by forgetting the contribution of player \II{} to the outcome of the adversarial Gowers' games and switching the roles of players \I{} and \II{}, we see that $\X$ is strategically Ramsey if and only if $\X'$ is adversarially Ramsey. So, for a suitable class $\Gamma$ of subsets of Polish spaces, saying that all $\Gamma$-subsets of $E^\omega$ are adversarially Ramsey is stronger than saying that all $\Gamma$-subsets of $E^\omega$ are strategically Ramsey. The second remark is that, if the field $K$ is infinite, then the adversarial Ramsey property for $\Gamma$-subsets of $E^\omega$ also implies that all $\Gamma$-subsets of $\omega^\omega$ are determined. To see this, observe that when playing vectors in $A_X$ or $B_X$, no matter the constraint imposed by the other player, players \I{} and \II{} have total liberty for choosing the first non-zero coordinate of the vectors they play. Therefore, by making $\X$ only depend on the first nonzero coordinate of each vector played, we recover a classical Gale-Stewart game in $(K^*)^\omega$. For this reason, there is no hope, in $ZFC$, to prove the adversarial Ramsey property for a class larger than Borel sets. Then, Rosendal asks the following questions in \cite{Rosendaladverse}:

\begin{question}[Rosendal]\label{mainvector}

Is every Borel set adversarially Ramsey?

\end{question}

\begin{question}[Rosendal]\label{mainvectoranalytic}

In the presence of large cardinals, is every analytic set adversarially Ramsey?

\end{question}

This paper is organised as follows. In section \ref{par2}, we define an abstract setting for infinite-dimensional Ramsey theory, called the setting of \emph{Gowers spaces}, where the asymptotic game, Gowers' game, and the adversarial Gowers' games can be defined. In this general setting, we prove that the answer to Question \ref{mainvector} is yes, and that the answer to Question \ref{mainvectoranalytic} is yes in the presence of a measurable cardinal. In section \ref{par3}, we prove a general version of Rosendal's Theorem \ref{RosendalGowersthm} in Gowers spaces; then, we introduce the pigeonhole principle in these spaces and we show that if this principle holds, then we can refine the abstract Rosendal's theorem to a result with an asymptotic game in both sides, which Mathias--Silver's theorem is an easy consequence. In section \ref{par4}, we give approximate versions of the results of sections \ref{par2} and \ref{par3} enabling to work in Polish metric spaces instead of countable spaces. Finally, in section \ref{par5}, we develop a general framework to deduce a genuine, non game-theoretical Ramsey conclusion (of the form ``in some subspace, every sequence satisfying some structural property is in $\X$'') from a conclusion involving a strategy of player \I{} in the asymptotic game. The results of this section, combined with these of section \ref{par4}, will for instance have Gowers' Theorems \ref{GowersThm} and \ref{Gowersc0theorem} as an immediate consequence.

\bigskip\bigskip

\section{Gowers spaces and the aversarial Ramsey\\property}\label{par2}

In this section, we will introduce the notion of a \emph{Gowers space}, which will be our abstract setting for infinite-dimensional Ramsey theory; then, we will prove in this setting the adversarial Ramsey principle, our most general Ramsey result without pigeonhole principle, which will give positive answers to Questions \ref{mainvector} and \ref{mainvectoranalytic}.

\smallskip

Inspired by the examples given in the introduction, we define a formalism with two notions, a notion of \emph{subspaces} and a notion of \emph{points}. The idea is that we will color infinite sequences of points and try to find subspaces such that lots of sequences of points in this subspace share the same color.

\smallskip

For $X$ a set, denote by $\Seq(X) = X^{< \omega} \setminus \{\varnothing\}$ the set of all nonempty finite sequences of elements of $X$. For $s, t \in X^{< \omega}$, denote by $|s|$ the length of $s$ (i.e. the integer $n$ such that $s \in X^n$), by $s \subseteq t$ the fact that $s$ is an initial segment of $t$, and by $s \concat t$ the concatenation of $s$ and $t$. If $x \in X$, the concatenation of $s$ and of the sequence $(x)$ whose unique term is $x$ will be abusively denoted by $s \concat x$.

\begin{defin}\label{DefGowersSpaces}

A \emph{Gowers space} is a quintuple $\G = (P, X, \leqslant, \leqslant^*, \vartriangleleft\nolinebreak)$, where $P$ is a nonempty set (the set of \emph{subspaces}), $X$ is an at most countable nonempty set (the set of \emph{points}), $\leqslant$ and $\leqslant^*$ are two quasiorders on $P$ (i.e. reflexive and transitive binary relations), and $\vartriangleleft \; \subseteq \Seq(X) \times P$ is a binary relation, satisfying the following properties:

\begin{enumerate}

\item for every $p, q \in P$, if $p \leqslant q$, then $p \leqslant^* q$;

\item for every $p, q \in P$, if $p \leqslant^* q$, then there exists $r \in P$ such that $r \leqslant p$, $r\leqslant q$ and $p \leqslant^* r$;

\item for every $\leqslant$-decreasing sequence $(p_i)_{i \in \omega}$ of elements of $P$, there exists $p^* \in P$ such that for all $i \in \omega$, we have $p^* \leqslant^* p_i$;

\item for every $p \in P$ and $s \in X^{<\omega}$, there exists $x \in X$ such that $s \concat x \vartriangleleft p$;

\item for every $s \in \Seq(X)$ and every $p, q \in P$, if $s \vartriangleleft p$ and $p\leqslant q$, then $s \vartriangleleft q$.

\end{enumerate}

Say that $p, q \in P$ are \emph{compatible} if there exists $r \in P$ such that $r \leqslant p$ and $r \leqslant q$. To save writing, we will often write $p \lessapprox q$ when $p \leqslant q$ and $q \leqslant^* p$. Observe that by 2., the $p^*$ in 3. can be chosen in such a way that $p^* \leqslant p_0$.

\end{defin}

In most usual cases, the fact that $s \vartriangleleft p$ will only depend on $p$ and on the last term of $s$; the spaces satisfying this property will be called \emph{forgetful Gowers spaces}. In these spaces, we will allow ourselves to view $\vartriangleleft$ as a binary relation on $X \times P$. However, for some applications (see, for example, the proof of Theorem \ref{ThmApproxRamsey}), it is sometimes useful to make the fact that $s \vartriangleleft p$ also depend on the the length of the sequence $s$, and we could imagine that even all the terms of $s$ would be relevant in some contexts.

\smallskip

When thinking about a Gowers space, we should have the two following examples in mind:

\begin{itemize}

\item The \emph{Mathias--Silver space} $\Sil = (\elinf, \omega, \subseteq, \subseteq^*, \vartriangleleft)$, where $\elinf$ is the set of all infinite sets of integers, $M \subseteq^* N$ iff $M\setminus N$ is finite and $(x_0, \ldots, x_n) \vartriangleleft M$ iff $x_n \in M$. Here, we have that $M \lessapprox N$ iff $M$ is a cofinite subset of $N$, and $M$ and $N$ are compatible iff $M \cap N$ in infinite.

\item The \emph{Rosendal space} over an at most countable field $K$, $\Ros_K = (P, E\setminus\{0\}, \subseteq, \subseteq^*\nolinebreak, \vartriangleleft\nolinebreak)$, where $E$ is a countably infinite-dimensional $K$-vector space with a basis $(e_i)_{i \in \omega}$, $P$ is the set of all block subspaces of $E$ relative to this basis, $X \subseteq^* Y$ iff $Y$ contains some finite-codimensional block subspace of $X$, and $(x_0, \ldots, x_n) \vartriangleleft X$ iff $x_n \in X$. Here, we have that $X \lessapprox Y$ iff $X$ is a finite-codimensional subspace of $Y$, and $X$ and $Y$ are compatible iff $X \cap Y$ is infinite-dimensional. 

\end{itemize}

Observe that both of these spaces are forgetful, so we could have defined $\vartriangleleft$ as a relation between points and subspaces (and that is what we will do, in such cases, in the rest of this paper); in this way, in both cases, $\vartriangleleft$ is the membership relation. It is easy to verify that, for these examples, the axioms 1., 2., 4., and 5. are satisfied; we briefly explain how to prove 3.. For the Mathias--Silver space, if $(M_i)_{i \in \omega}$ is a $\subseteq$-decreasing sequence of infinite subsets of $\omega$, then we can, for each $i \in \omega$, choose $n_i \in M_i$ in such a way that the sequence $(n_i)_{i \in \omega}$ is increasing, and let $M^* = \{n_i \mid i \in \omega\}$. Then the set $M^*$ is as wanted. For the Rosendal space, the idea is the same: given $(F_i)_{i \in \omega}$ a decreasing sequence of block subspaces of $E$, we can pick, for each $i$, a nonzero vector $x_i \in F_i$, in such a way that for $i \geqslant 1$, we have $\supp(x_{i - 1}) < \supp(x_i)$. In this way, $(x_i)_{i \in \omega}$ is a block sequence, and the block subspace $F^*$ spanned by this sequence is as wanted.

\smallskip

Also observe that in the definition of the Rosendal space, choosing $E\setminus\{0\}$ and not $E$ for the set of points is totally arbitrary, and here, we only made this choice in order to use the same convention as Rosendal in his papers \cite{RosendalGowers, Rosendaladverse}; but the results we will show apply as well when the set of points is $E$. Also, we could have taken for $P$ the set of all infinite-dimensional subspaces of $E$ (where, here, the relation $\subseteq^*$ is defined by $X \subseteq^* Y$ iff $X \cap Y$ has finite codimension in $X$) instead of only block subspaces. However, the abstract results we will prove are slightly stronger in the case when we consider only block subspaces; this is due to the fact that, while every infinite-dimensional subspace of $E$ contains a block subspace, there are finite-codimensional subspaces that do not contain any finite-codimensional block subspace.

\smallskip

In the rest of this section, we fix a Gowers space $\G = (P, X, \leqslant, \leqslant^*, \vartriangleleft)$. For $p \in P$, define the \emph{adversarial Gowers' games below $p$} as follows:

\begin{defin}\label{DefAdversarialGames}

\nopagebreak

\alaligne

\nopagebreak

\begin{enumerate}

\item The game $A_p$ is defined in the following way:

\smallskip

\begin{tabular}{cccccccc}
\textbf{I} & & $x_0,  q_0$ & & $x_1,  q_1$ & & $\hdots$ \\
\textbf{II} & $p_0$ & & $y_0,  p_1$ & & $y_1,  p_2 $ & & $\hdots$ 
\end{tabular}

\smallskip

\noindent where the $x_i$'s and the $y_i$'s are elements of $X$, and the $p_i$'s and the $q_i$'s are elements of $P$. The rules are the following:

\begin{itemize}

\item for \textbf{I}: for all $i \in \omega$, $(x_0, y_0, \ldots, x_{i - 1}, y_{i - 1}, x_i) \vartriangleleft p_i$ and $q_i \lessapprox p$;

\item for \textbf{II}: for all $i \in \omega$, $(x_0, y_0 \ldots, x_i, y_i) \vartriangleleft q_i$ and $p_i \leqslant p$.

\end{itemize}

\noindent The outcome of the game is the sequence $(x_0, y_0, x_1, y_1, \ldots) \in X^\omega$.

\item The game $B_p$ is defined in the same way as $A_p$, except that this time the we require $p_i \lessapprox p$, whereas we only require $q_i \leqslant p$.

\end{enumerate}

\end{defin}

As in the particular case of vector spaces, we can define the \emph{adversarial Ramsey property} for subsets of $X^\omega$:

\begin{defin}

A set $\X \subseteq X^\omega$ is said to be \emph{adversarially Ramsey} if for every $p \in P$, there exists $q \leqslant p$ such that either player \I{} has a strategy to reach $\X$ in $A_q$, or player \II{} has a strategy to reach $\X^c$ in $B_q$.

\end{defin}

Informally, the adversarial Ramsey property for $\X$ means that up to taking a subspace, one of the players has a winning strategy in the game that is the most difficult for him. Observe that the property that \I{} has a strategy in $A_p$ to reach some set $\X$ (resp. the property that \II{} has a strategy in $B_p$ to reach $\X^c$) is strongly hereditary in the sense that if \I{} has a strategy to reach $\X$ in $A_p$, then he also has one in $A_{p'}$ for every $p' \leqslant^* p$ (and the same holds for \II{} in $B_p$). Indeed, we can simulate a play of $A_{p'}$ with a play of $A_p$: when, in $A_p$, player \I{}'s strategy tells him to play $x_i$ and $q_i$, then in $A_{p'}$ he can play the same $x_i$ and a $q_i'$ such that $q'_i \lessapprox p'$ and $q'_i \leqslant q_i$, in such a way that the next $y_i$ played by \II{} in $A_{p'}$ will be also playable in $A_p$ (the existence of such a $q'_i$ is guaranteed by condition 2. in the definition of a Gowers space). And when, in $A_{p'}$, player \II{} plays $y_i$ and $p'_{i + 1}$, then in $A_p$, \I{} can make her play the same $y_i$ and a $p_{i + 1}$ such that $p_{i + 1} \leqslant p$ and $p_{i + 1} \leqslant p'_{i + 1}$, in such a way that the next $x_{i + 1}$ played by \I{} in $A_p$ according to his strategy will also be playable in $A_{p'}$. In this way, the outcomes of both games are the same, and since \I{} reaches $\X$ in $A_p$, then he also does in $A_{p'}$.

\smallskip

On the other hand, it is clear that if $\I{}$ has a strategy to reach some set $\X$ in $A_p$, then he also has one in $B_p$, so \II{} cannot have a strategy to reach $\X^c$ in $B_p$. Thus, the fact that $\X$ has the adversarial Ramsey property gives a genuine dichotomy between two disjoint and strongly hereditary classes of subspaces.

\smallskip

We endow the set $X$ with the discrete topology and the set $X^\omega$ with the product topology. The main result of this section is the following:

\begin{thm}[Adversarial Ramsey principle, abstract version]\label{main}

Every Borel subset of $X^\omega$ is adversarially Ramsey.

\end{thm}

In the case of the Rosendal space, the adversarial Gowers games defined here are exactly the same as those defined in the introduction. Thus, Theorem \ref{main} applied to this space provides a positive answer to Question \ref{mainvector}.

\smallskip

Also observe that if $P = \{\mathds{1}\}$ and if we have $s \vartriangleleft \mathds{1}$ for every $s \in \Seq(X)$, then both $A_\mathds{1}$ and $B_\mathds{1}$ are the classical Gale-Stewart game in $X$, so the adversarially Ramsey subsets of $X^\omega$ are exactly the determined ones. So in this space, Theorem \ref{main} is nothing more than Borel determinacy for games on integers; hence, we get that Theorem \ref{main} has at least the metamathematical strength of Borel determinacy for games on integers. Therefore, by the work of Friedman \cite{Friedman}, any proof of Theorem \ref{main} should make use of the powerset axiom and of the replacement scheme. We also get that it is not provable in $ZFC$ that every analytic (or coanalytic) set in every Gowers space is adversarially Ramsey. Actually, it turns out that there is a large class of Gowers spaces for which Borel determinacy can be recovered from the version of Theorem \ref{main} in these spaces; this will be shown in a forthcoming paper \nolinebreak \cite{RancourtRamseyII}.

\smallskip

We will deduce Theorem \ref{main} from Borel determinacy for games on real numbers. For this purpose, we follow an approach used first by Kastanas in \cite{kastanas}: in this paper Kastanas deduced the Ramsey property for subsets of $\elinf$ from the determinacy of a game 
(a similar game was already introduced by Mansfield in \cite{Mansfield}, although it was only used in the open case, in order to estimate the complexity of a homogeneous set). In what follows, we adapt Kastanas' game in order to get the adversarial Ramsey property.

\begin{defin}

For $p \in P$, \emph{Kastanas' game} $K_p$ below $p$ is defined as follows:

\smallskip

\begin{tabular}{cccccccc}
\textbf{I} & & $x_0,  q_0$ & & $x_1,  q_1$ & & $\hdots$ \\
\textbf{II} & $p_0$ & & $y_0,  p_1$ & & $y_1,  p_2 $ & & $\hdots$ 
\end{tabular}

\smallskip

\noindent where the $x_i$'s and the $y_i$'s are elements of $X$, and the $p_i$'s and the $q_i$'s are elements of $P$. The rules are the following:

\begin{itemize}

\item for \textbf{I}: for all $i \in \omega$, $(x_0, y_0, \ldots, x_{i - 1}, y_{i - 1}, x_i) \vartriangleleft p_i$ and $q_i \leqslant p_i$;

\item for \textbf{II}: $p_0 \leqslant p$, and for all $i \in \omega$, $(x_0, y_0 \ldots, x_i, y_i) \vartriangleleft q_i$ and  $p_{i+1} \leqslant q_i$.

\end{itemize}

\noindent The outcome of the game is the sequence $(x_0, y_0, x_1, y_1, \ldots) \in X^\omega$.

\end{defin}

The exact result we will show is the following:

\begin{prop}\label{mainaux}

Let $p \in P$ and $\X \subseteq X^\omega$.

\begin{enumerate}

\item If \I{} has a strategy to reach $\X$ in $K_p$, then there exists $q\leqslant p$ such that \I{} has a strategy to reach $\X$ in $A_q$;

\item If \II{} has a strategy to reach $\X^c$ in $K_p$, then there exists $q\leqslant p$ such that \II{} has a strategy to reach $\X^c$ in $B_q$.

\end{enumerate}

\end{prop}

Once this proposition is proved, Theorem \ref{main} will immediately follow from the Borel determinacy of Kastanas' game.

\smallskip

Since the proofs of \emph{1.} and \emph{2.} of Proposition \ref{mainaux} are exactly similar, we only prove \emph{2.}. In order to do this, let us introduce some notation. During the whole proof, we fix a strategy $\tau$ for \II{} in $K_p$ to reach $\X^c$. Following the terminology introduced by Ferenczi and Rosendal in \cite{fero}, a partial play of $K_p$ ending with a move of \II{} and during which \II{} always plays according to her strategy will be called a \emph{state}. Say that a state $\st$ \emph{realises} a finite sequence $(x_0, y_0, x_1, y_1, \ldots, x_{n - 1}, y_{n - 1})$ if $\st$ has the form $(p_0, x_0, ..., q_{n - 1}, y_{n - 1}, p_n)$; say that a state realising a sequence of length $2n$ \emph{has rank $n$}. Define in the same way the notions of a \emph{total state} (which is a total play of $K_p$) and of realisation for a total state; the restriction of a total state $\st =  (p_0, x_0, q_0, y_0, p_1, ...)$ to a state of rank $n$, denoted by $\st_{\restriction n}$, is the state $(p_0, x_0, ..., q_{n - 1}, y_{n - 1}, p_n)$. If an infinite sequence $(x_0, y_0, x_1, y_1, \ldots)$ is realised by a total state, then this sequence belongs to $\X^c$.

\smallskip

We will use the following lemma:

\begin{lem}\label{mainlem}

Let $\St$ be an at most countable set of states, and $r \in P$. Then there exists $r^* \leqslant r$ satisfying the following property: for all $\st \in \St$ and $x, y \in X$ if there exists $u, v \in P$ such that:

\begin{enumerate}

\item \I{} can legally continue the play $\st$ by the move $(x, u)$;

\item $\tau(\st \concat (x, u)) = (y, v)$;

\item $v$ and $r^*$ are compatible;

\end{enumerate}

\noindent then there exists $u', v' \in P$ satisfying \emph{1.}, \emph{2.}, and \emph{3.} and such that, moreover, we have $r^* \leqslant^* v'$.

\end{lem}

\begin{proof}

Let $(\st_n, x_n, y_n)_{n \in \omega}$ be a (non-necessarily injective) enumeration of $\St \times X^2$. Define $(r_n)_{n \in \omega}$ a decreasing sequence of elements of $P$ in the following way. Let $r_0 = r$. For $n \in \omega$, suppose $r_n$ defined. If there exists a pair $(u, v) \in P^2$ such that:

\begin{itemize}

\item \I{} can legally continue the play $\st_n$ by the move $(x_n, u)$;

\item $\tau(\st_n \concat (x_n, u)) = (y_n, v)$;

\item $v$ and $r_n$ are compatible;

\end{itemize}

\noindent then choose $(u_n, v_n)$ such a pair and let $r_{n + 1}$ be a common lower bound to $r_n$ and $v_n$. Otherwise, let $r_{n + 1} = r_n$. This achieves the construction.

\smallskip

By the definition of a Gowers space, there exists $r^* \in P$ such that $r^* \leqslant r$ and for all $n \in \omega$, $r^* \leqslant^* r_n$. We show that $r^*$ is as required. Let $n \in \omega$, and suppose that there exists $(u, v) \in P^2$ satisfying properties \emph{1.}, \emph{2.}, and \emph{3.} as in the statement of the lemma for the triple $(\st_n, x_n, y_n)$. Since $r^* \leqslant^* r_n$ and since $v$ and $r^*$ are compatible, then $v$ and $r_n$ are also compatible. This show that the pair $(u_n, v_n)$ has been defined; by construction, this pair satisfies properties \emph{1.} and \emph{2.} for $(\st_n, x_n, y_n)$, and we have $r_{n + 1} \leqslant v_n$, so $r^* \leqslant ^* v_n$, which shows that $(u', v') = (u_n, v_n)$ is as required.

\end{proof}

\begin{proof}[Proof of Proposition \ref{mainaux}]

Define $(q_n)_{n \in \omega}$ a decreasing sequence of elements of $P$ and $(\St_n)_{n \in \omega}$ a sequence where, for every $n \in \omega$, $\St_n$ is an at most countable set of states of rank $n$, in the following way. Let $q_0 = \tau(\varnothing)$ and $\St_0 = \{(\tau(\varnothing))\}$. For $n \in \omega$, suppose $q_n$ and $\St_n$ being defined. Let $q_{n + 1}$ be the result of the application of Lemma \ref{mainlem} to $q_n$ and the set of states $\St_n$. For $\st \in \St_n$, let $\Ast$ be the set of all pairs $(x, y)$ such that there exists $(u, v) \in P^2$ satisfying:

\begin{enumerate}

\item \I{} can legally continue the play $\st$ by the move $(x, u)$;

\item $\tau(\st \concat (x, u)) = (y, v)$;

\item $v$ and $q_{n + 1}$ are compatible.

\end{enumerate}

\noindent Then by construction of $q_{n + 1}$, for all $(x, y) \in \Ast$, there exists a pair $(u, v) \in P^2$ satisfying 1., 2., and 3., and such that moreover $q_{n + 1} \leqslant^* v$. For each $(x, y) \in \Ast$, choose $(u_{\stp, x, y}, v_{\stp, x, y})$ such a pair. Let $\St_{n + 1} = \{\st \concat (x, u_{\stp, x, y}, y, v_{\stp, x, y}) \, | \, \st \in \St_n, \; (x, y) \in \Ast\}$; this is clearly a countable set of states of rank $n + 1$. This achieves the construction.

\smallskip

Now let $q \in P$ be such that $q \leqslant q_0$ and for all $n \in \omega$, we have $q \leqslant^* q_n$. Observe that since $q_0 \leqslant p$, we have $q \leqslant p$. We show that $q$ is as required, by describing a strategy for \II{} in $B_q$ to reach $\X^c$. In order to do this, we simulate the play $\st = (v_0, x_0, u_0, y_0, v_1, ...)$ of $B_q$ that \I{} and \II{} are playing by a play $\st' = (v'_0, x_0, u'_0, y_0, v'_1, ...)$ of $K_p$ having the same outcome and during which \II{} always plays according to her strategy $\tau$. This will ensure that the outcome $(x_0, y_0, x_1, y_1, \ldots)$ of both games lies in $\X^c$ and so that the strategy for \II{} in $B_q$ that we described enables her to reach her goal. We do this construction in such a way that at each turn $n$, the following conditions are kept satisfied:

\begin{enumerate}

\item[(a)] $\st'_{\restriction n} \in \St_n$;

\item[(b)] $v_n \leqslant v'_n$.

\end{enumerate}

The moves of the players at the $(n+1)^\text{th}$ turn in both games that are described in the following proof are represented in the diagrams below. The third diagram, called ``Fictive $K_p$'', represents a fictive situation that will be studied for technical reasons in the proof, and in which the moves of both players are the same as in $K_p$ until the $n^\text{th}$ turn but differ from the $(n + 1)^\text{th}$ turn.

\bigskip

\begin{tabular}{ccccccccc}
& & \textbf{I} & $\hdots$ & & & $x_n, \, u_n$ & &  \\
$B_q$ & & & & & & \\
& & \textbf{II} & & $\hdots, \, v_n$ & & & & $y_n,  \, v_{n + 1}$ \\
\\
\\
\\
& & \textbf{I} & $\hdots$ & & & $x_n, \, u'_n$ & &  \\
$K_p$ & & & & & & \\
& & \textbf{II} & & $\hdots, \, v'_n$ & & & & $y_n,  \, v'_{n + 1}$ \\
\\
\\
\\
& & \textbf{I} & $\hdots$ & & & $x_n, \, u''_n$ & &  \\
Fictive $K_p$ & & & & & & \\
& & \textbf{II} & & $\hdots, \, v'_n$ & & & & $y_n,  \, v''_{n + 1}$
\end{tabular}

\bigskip

Let us describe the strategy of \II{} in $B_q$. At the first turn, this strategy will consist in playing $v_0 = q$; and, according to her strategy $\tau$, \II{} will play $v'_0 = \tau(\varnothing)$ in $K_p$. Now, suppose that both games have been played until the $n$\textsuperscript{th} turn, that is, the last moves of player \II{} in the games $B_q$ and $K_p$ are respectively $v_n$ and $v'_n$.  Player \I{} plays $(x_n, u_n)$ in $B_q$. By the rules of the game $B_q$ and the induction hypothesis, we have that $u_n \leqslant q \leqslant^* v_n \leqslant v'_n$; so there exists $u''_n \in P$ such that $u''_n \lessapprox u_n$ and $u''_n \leqslant v'_n$. We also have that $(x_0, y_0, \ldots, x_n) \vartriangleleft v_n \leqslant v'_n$, so it is legal for \I{} to pursue the game $K_p$ by playing $(x_n, u''_n)$; this fictive situation is represented in the third diagram above, called ``Fictive $K_p$''. In this fictive situation, the strategy $\tau$ of \II{} would lead her to answer with a move $(y_n, v''_{n + 1})$ satisfying $(x_0, y_0, \ldots, x_n, y_n) \vartriangleleft u''_n$ and $v''_{n + 1} \leqslant u''_n$. We have, by construction of $q$, that $v''_{n + 1} \leqslant u''_n \leqslant u_n \leqslant q \leqslant^* q_{n + 1}$; so in particular, $v''_{n + 1}$ and $q_{n + 1}$ are compatible. Recalling that $\st_{\restriction n} \in \St_n$, we see that the pair $(u''_n, v''_n)$ witnesses that $(x_n, y_n) \in \Asti{\restriction n}$.

\smallskip

Now let us leave the fictive situation and come back to the ``real'' $K_p$. Since $(x_n, y_n) \in \Asti{\restriction n}$, we know that the pair $(u_{\stp_{\restriction n}, x_n, y_n}, v_{\stp_{\restriction n}, x_n, y_n})$ has been defined; we denote this pair by $(u'_n, v'_{n + 1})$. In the ``real'' $K_p$, we make \I{} play $(x_n, u'_n)$. By definition of $(u'_n, v'_{n + 1})$, this move is legal, and \II{} will answer, according to her strategy, with $(y_n, v'_{n + 1})$. Observe that the required condition (a) in the induction hypothesis is satisfied by these moves since, by the definition of $\St_{n + 1}$, we have $\st_{\restriction n} \concat (x_n, u'_n, y_n, v'_{n + 1}) \in \St_{n + 1}$. We also have that $q \leqslant^* q_{n + 1} \leqslant^* v'_{n + 1}$, so there exists $v_{n + 1} \in P$ such that $v_{n + 1} \leqslant v'_{n +1}$ and $v_{n + 1} \lessapprox q$. For this reason, and since we also have (as we already saw) $(x_0, y_0, \ldots, x_n, y_n) \vartriangleleft u''_n \leqslant u_n$, we get that $(y_n, v_{n + 1})$ is a legal move for \II{} in $B_q$, that satisfies the condition (b) in the induction hypothesis. So we just have to define her strategy as leading her to play this move, and this achieves the proof.

\nopagebreak

\end{proof}

We actually proved a little more than Theorem \ref{main}. Say that the Gowers space $\G$ is \emph{analytic} if $P$ is an analytic subset of a Polish space and if the relations $\leqslant$ and $\vartriangleleft$ are Borel subsets of $P^2$ and of $\Seq(X) \times P$ respectively. For most of the spaces we actually use, $P$ can be identified to an analytic subset of $\mathcal{P}(X)$, the relation $\leqslant$ to the inclusion, and the relation $(x_0, \ldots, x_n) \vartriangleleft p$ to the membership relation $x_n \in p$
; thus, these spaces are analytic. This is, for instance, the case for the Mathias--Silver space and the Rosendal space introduced at the beginning of this section.

\smallskip

Now, say that a class $\Gamma$ of subsets of Polish spaces is \emph{suitable} if it contains the class of Borel sets and is stable under finite unions, finite intersections and Borel inverse images. Equip $\R$ with its usual Polish topology, and $\R^\omega$ with the product topology. Then an easy consequence of Proposition \ref{mainaux} is the following:

\begin{cor}\label{AdvGamma}

Let $\Gamma$ be a suitable class of subsets of Polish spaces. If every $\Gamma$-subset of $\R^\omega$ is determined, then for an analytic Gowers space $\G = (P, X, \leqslant, \leqslant^*, \vartriangleleft)$, every $\Gamma$-subset of $X^\omega$ is adversarially Ramsey.

\end{cor}

\begin{proof}

Fix $\X \subseteq X^\omega$ a $\Gamma$-subset, and $p \in P$. By Proposition \ref{mainaux}, it is enough to show that in the game $K_p$, either player \I{} has a strategy to reach $\X$, or player \II{} has a strategy to reach $\X^c$. Let $\varphi : \R \longrightarrow P$ be a surjective Borel mapping, and consider the following game $K'_p$:

\smallskip

\begin{tabular}{cccccccc}
\textbf{I} & & $x_0,  \widetilde{q_0}$ & & $x_1,  \widetilde{q_1}$ & & $\hdots$ \\
\textbf{II} & $\widetilde{p_0}$ & & $y_0,  \widetilde{p_1}$ & & $y_1,  \widetilde{p_2}$ & & $\hdots$ 
\end{tabular}

\smallskip

\noindent where the $x_i$'s and the $y_i$'s are elements of $X$ and the $\widetilde{p_i}$'s and the $\widetilde{q_i}$'s are real numbers, with the constraint that $\varphi(\widetilde{p_0}) \leqslant p$, for all $i \in \omega$, $\varphi(\widetilde{q_i}) \leqslant \varphi(\widetilde{p_i})$, $\varphi(\widetilde{p_{i + 1}}) \leqslant \varphi(\widetilde{q_i})$, $(x_0, y_0, \ldots, x_i) \vartriangleleft \varphi(\widetilde{p_i})$, and $(x_0, y_0, \ldots, x_i, y_i) \vartriangleleft \varphi(\widetilde{q_i})$, and whose outcome is the sequence $(x_0, y_0, x_1, y_1, \ldots) \in X^\omega$. This game is clearly equivalent to $K_p$: \I{} has a strategy to reach $\X$ in $K_p$ if and only if he has one in $K'_p$, and \II{} has a strategy to reach $\X^c$ in $K_p$ if and only if she has one in $K'_p$. Since $K'_p$ is a game on real numbers with Borel rules and since $\X$ is in $\Gamma$, we deduce that in this game, either \I{} has a strategy to reach $\X$, or \II{} has a strategy to reach $\X^c$, what concludes the proof.

\end{proof}

Martin proved in \cite{MartinMeasurable} that if there exists a measurable cardinal, then every analytic subset of $\omega^\omega$ is determined. He observed in this paper that his proof actually shows a little more: if $\kappa$ is a measurable cardinal, if $S$ is a set with the discrete topology such that $|S| < \kappa$ and if $S^\omega$ is endowed with the product topology, then every $\mathbf{\Sigma}_1^1$-subset of $S^\omega$ is determined. (Here, a $\mathbf{\Sigma}_1^1$-subset of a topological space $X$ is defined as a set $A \subseteq X$ which is the first projection of some closed subset of $X \times \omega^\omega$.) In particular, if there exists a measurable cardinal, then every analytic subset of $\R^\omega$ is determined, and by the last corollary, every analytic set in an analytic Gowers space is adversarially Ramsey. This gives an answer to Question \ref{mainvectoranalytic}.

\smallskip

Recall that $PD_\R$ is the statement ``every projective subset of $\R^\omega$ (with $\R$ endowed with its usual Polish topology) is determined'' and that $AD_\R$ is the statement ``every subset of $\R^\omega$ is determined''. Corollary \ref{AdvGamma} shows in particular that, in an analytic Gowers space, under $PD_\R$, every projective set is adversarially Ramsey. Recall that Harrington and Kechris \cite{harringtonkechris}, and independently Woodin \cite{Woodin} proved that under $PD$, every projective subset of $\elinf$ is Ramsey. Using ideas from Woodin's proof, Bagaria and L\'opez-Abad \cite{BagariaLopezAbad2} showed that under $PD$, every projective set of block sequences of a basis of a Banach space is strategically Ramsey (i.e. satisfies the conclusion of Gowers' Theorem \ref{GowersThm}). Basing ourselve on these facts, we can formulate the following conjecture:

\begin{conj}

Under $PD$, if the Gowers space $\G = (P, X, \leqslant, \leqslant^*, \vartriangleleft)$ is analytic, then every projective subset of $X^\omega$ is adversarially Ramsey.

\end{conj}

Clearly, the method presented in the present paper does not enable to prove this.

\smallskip

Also observe that the proof of Proposition \ref{mainaux} can almost entierly be done in $ZF + DC$; the only use of the full axiom of choice is made to choose $u''_n \in P$ such that $u''_n \lessapprox u_n$ and $u''_n \leqslant v'_n$, and $v_{n + 1} \in P$ such that $v_{n + 1} \leqslant v'_{n +1}$ and $v_{n + 1} \lessapprox q$, so actually to apply axiom 2. in the definition of a Gowers space. For this reason, say that the Gowers space $\mathcal{G}$ is \emph{effective} if in this axiom 2., the subspace $r$ can be chosen in an effective way, that is, if there exist a function $f : P^2 \longrightarrow P$ such that for every $p, q \in P$, if $p \leqslant^* q$, then we have $f(p, q) \lessapprox p$ and $f(p, q) \leqslant q$. For instance:

\begin{itemize}

\item The Mathias--Silver space is effective: indeed, if $M \subseteq^* N$, then we can take $f(M, N) = M \cap N$.

\item The Rosendal space is effective. Indeed, if $X$ and $Y$ are block subspaces such that $X \subseteq^* Y$, let $(x_n)_{n \in \omega}$ be a block sequence spanning $X$. Then we can let $f(X, Y)$ be the subspace spanned by the largest final segment of $(x_n)$ all of whose terms are in $Y$ (this subspace does not depend on the choice of $(x_n)$).

\end{itemize}

To prove Proposition \ref{mainaux} for an effective Gowers space, we only need dependant choices. Thus, we have the following result:

\begin{cor}[$ZF + DC + AD_\R$]\label{AdvAD}

Let $\G = (P, X, \leqslant, \leqslant^*, \vartriangleleft)$ be an effective Gowers space such that $P$ is a subset of a Polish space. Then every subset of $X^\omega$ is adversarially Ramsey.

\end{cor}

\begin{proof}

Recall that in $ZF + DC + AD$, every subset of a Polish space is either at most countable, or contains a Cantor set, and is thus in bijection with $\R$ (this is a consequence of Theorem 21.1 in \cite{kechris}, that can be proved in $ZF + DC$). So if $P$ is countable, then Kastanas' game can be viewed as a game on integers and is thus determined, and if $P$ is uncountable, then Kastanas' game can be viewed as a game on real numbers, that is also determined. The conclusion follows from Proposition \nolinebreak \ref{mainaux}.

\end{proof}

As above, we can't prove in this way that the same result holds under $AD$ instead of $AD_\R$, but we conjecture that it does so. As we will see in the next section, if this is true, this would imply that under $AD$, every subset of $\elinf$ is Ramsey, which is still conjectural today. More generally, it would be interesting to know whether, in general, the determinacy of $\Gamma$-subsets of $\omega^\omega$ is enough to prove the adversarial Ramsey property for $\Gamma$-sets in sufficiently regular Gowers spaces, for any suitable class $\Gamma$.







\bigskip

\section{Strategically Ramsey sets and the pigeonhole principle}\label{par3}

The aim of this section is to prove a version of Rosendal's Theorem \ref{RosendalGowersthm} in the general setting of Gowers spaces. We also introduce the notion of the \emph{pigeonhole principle} for a Gowers space and see that the latter result can be strengthened in the case where this principle holds. This will enable us to see the fundamental difference between the Mathias--Silver space and the Rosendal space over a field with at least three elements. We start by introducing Gowers' game and the asymptotic game in the setting of Gowers spaces, and the notion of a strategically Ramsey set. In this whole section, we fix a Gowers space $\G = (P, X, \leqslant, \leqslant^*, \vartriangleleft)$.

\begin{defin}\label{DefGowersGames}

Let $p \in P$.

\begin{enumerate}

\item \emph{Gowers' game below $p$}, denoted by $G_p$, is defined in the following way:

\smallskip

\begin{tabular}{ccccccc}
\textbf{I} & $p_0$ & & $p_1$ & & $\hdots$ & \\
\textbf{II} & & $x_0$ & & $x_1$ & & $\hdots$ 
\end{tabular}

\smallskip

\noindent where the $x_i$'s are elements of $X$, and the $p_i$'s are elements of $P$. The rules are the following:

\begin{itemize}

\item for \textbf{I}: for all $i \in \omega$, $p_i \leqslant p$;

\item for \textbf{II}: for all $i \in \omega$, $(x_0, \ldots, x_i) \vartriangleleft p_i$.

\end{itemize}

\noindent The outcome of the game is the sequence $(x_i)_{i \in \omega} \in X^\omega$.

\item The \emph{asymptotic game below $p$}, denoted by $F_p$, is defined in the same way as $G_p$, except that this time we moreover require that $p_i \lessapprox p$.

\end{enumerate}

\end{defin}

\begin{defin}

A set $\X \subseteq X^\omega$ is said to be \emph{strategically Ramsey} if for every $p \in P$, there exists $q \leqslant p$ such that either player \I{} has a strategy to reach $\X^c$ in $F_q$, or player \II{} has a strategy to reach $\X$ in $G_q$.

\end{defin}

The general version of Rosendal's Theorem \ref{RosendalGowersthm} is then the following:

\begin{thm}[Abstract Rosendal's theorem]\label{abstractRosendal}

Every analytic subset of $X^\omega$ is strategically Ramsey.

\end{thm}

Observe that Theorem \ref{RosendalGowersthm} is exactly the result of the application of Theorem \ref{abstractRosendal} to the Rosendal space.

\begin{proof}

We prove the result for Borel sets first. In order to do this, consider another space $\widetilde{\G} = (P, X, \leqslant, \leqslant^*, \widetilde{\vartriangleleft})$, where $P$, $X$, $\leqslant$, and $\leqslant^*$ are the same as in $\G$, but we replace $\vartriangleleft$ by the relation $\widetilde{\vartriangleleft}$ defined by $(x_0, y_0, x_1, y_1, \ldots, x_n, y_n) \, {\widetilde{\vartriangleleft}} \, p$ iff $(y_0, y_1, \ldots, y_n) \vartriangleleft p$, and $(x_0, y_0, x_1, y_1, \ldots, x_n) \, {\widetilde{\vartriangleleft}} \, p$ iff $(x_0, x_1, \ldots, x_n) \vartriangleleft p$. Now, to each set $\X \subseteq X^\omega$, associate a set $\widetilde{\X} \subseteq X^\omega$ defined by $(x_0, y_0, x_1, y_1, \ldots) \in \widetilde{\X} \Leftrightarrow (y_0, y_1, \ldots) \in \X$. Then, when players try to reach $\widetilde{\X}$ or $\widetilde{\X}^c$ in the games $A_p$ and $B_p$ of $\widetilde{\G}$, the $p_i$'s played by \II{} and the $x_i$'s played by \I{} don't matter at all; so a strategy for \I{} in the game $A_p$ of $\widetilde{\G}$ to reach $\widetilde{\X}^c$ becomes a strategy for \I{} in the game $F_p$ of $\G$ to reach $\X^c$, and a strategy for \II{} in the game $B_p$ of $\widetilde{\G}$ to reach $\widetilde{\X}$ becomes a strategy for \II{} in the game $G_p$ of $\G$ to reach $\X$. Thus, the strategical Ramsey property for $\X$ in $\G$ is equivalent to the adversarial Ramsey property for $\widetilde{\X}^c$ in $\widetilde{\G}$, so the strategical Ramsey property for Borel sets in $\G$ follows from Theorem \ref{main}.

\smallskip

From the result for Borel sets, we now deduce the result for arbitrary analytic sets using an unfolding argument (see \cite{kechris}, section 21, for a general presentation of the method of unfolding). Let $\X \subseteq X^\omega$ be analytic, and $p \in P$. Let $X' = X \times \nolinebreak \{0, 1\}$, whose elements will be denoted by the letters $(x, \varepsilon)$. Define the binary relation $\vartriangleleft' \subseteq \Seq(X') \times P$ by $(x_0, \varepsilon_0, \ldots, x_n, \varepsilon_n) \vartriangleleft' p$ if $(x_0, \ldots, x_n) \vartriangleleft p$, and consider the Gowers space $\G' = (P, X', \leqslant, \leqslant^*, \vartriangleleft')$. In this proof, we will use the notations $F_q$ and $G_q$ to denote respectively the asymptotic game and Gowers' game in the space $\G$, whereas the notations $F'_q$ and $G'_q$ will be used for these games in the space $\G'$. We denote by $\pi$ the projection $X'^\omega \longrightarrow X^\omega$. Let $\X' \subseteq X'^\omega$ be a $G_\delta$ set such that $\X = \pi(\X')$. Since $\X'$ is $\G_\delta$, it is strategically Ramsey by the first part of this proof; let $q \leqslant p$ witnessing so. If player \II{} has a strategy in $G'_q$ to reach $\X'$, then a run of the game $G_q$ where \II{} uses this strategy but omits to display the $\varepsilon_i$'s produces an outcome lying in $\X$; hence, \II{} has a strategy to reach $\X$ in $G_q$. Then, our result will follow from the following fact:

\begin{fact}\label{Fact34}

If \I{} has a strategy to reach $\X'^c$ in $F'_q$, then he has a strategy to reach $\X^c$ in $F_q$.

\end{fact}

\begin{proof}

Let $\tau'$ be a strategy enabling \I{} to reach $\X'^c$ in $F'_q$. In order to save notation, in this proof, we consider that in the games $F'_q$ and $F_q$, player \II{} is allowed not to respect the rules (i.e. to play $x_i$'s such that $(x_0, \ldots, x_i) \ntriangleleft p_i$), but loses the game if she does. Then, the strategy $\tau'$ can be viewed as a mapping $X'^{< \omega} \longrightarrow P$ such that for every $(x_0, \varepsilon_0, \ldots, x_{n - 1}, \varepsilon_{n - 1}) \in X'^{< \omega}$, we have $\tau'(x_0, \varepsilon_0, \ldots, x_{n - 1}, \varepsilon_{n - 1}) \lessapprox q$. Observe that if $(p_j)_{j \in J}$ is a finite family of elements of $P$ such that $\forall j \in J, \; p_j \lessapprox q$, then by applying iteratively the property 2. in the definition of a Gowers space, we can get $p^* \in P$ such that $p^* \lessapprox q$ and $\forall j \in J \; p^* \leqslant p_j$. Thus, for every $(x_0, \ldots, x_{n - 1}) \in X^{< \omega}$, we can choose $\tau(x_0, \ldots, x_{n - 1}) \in P$ such that $\tau(x_0, \ldots, x_{n - 1}) \lessapprox q$ and such that for every $(\varepsilon_0, \ldots, \varepsilon_{n - 1}) \in \{0, 1\}^n$, we have $\tau(x_0, \ldots, x_{n - 1}) \leqslant \tau'(x_0, \varepsilon_0 \ldots, x_{n - 1}, \varepsilon_{n - 1})$. We have hence defined a mapping $\tau : X^{< \omega} \longrightarrow P$; we show that this is a strategy for \I{} in $F_q$ enabling him to reach $\X^c$.

\smallskip

Consider a run of the game $F_q$ during which \II{} respects the rules and \I{} plays according to his strategy $\tau$:

\smallskip

\begin{tabular}{ccccccc}
\textbf{I} & $p_0$ & & $p_1$ & & $\hdots$ & \\
\textbf{II} & & $x_0$ & & $x_1$ & & $\hdots$ 
\end{tabular}

\smallskip

\noindent We have to show that $(x_i)_{i \in \omega} \notin \X$, that is, for every $(\varepsilon_i)_{i \in \omega} \in \{0, 1\}^\omega$, $(x_i, \varepsilon_i)_{i \in \omega} \notin \X'$. Let $(\varepsilon_i)_{i \in \omega} \in \{0, 1\}^\omega$; it is enough to show that $(x_i, \varepsilon_i)_{i \in \omega}$ is the outcome of a run of the game $F'_q$ during which \I{} always follows his strategy $\tau'$ and \II{} always respects the rules. Letting $p'_i$ = $\tau'(x_0, \varepsilon_0, \ldots, x_{n - 1}, \varepsilon_{n - 1})$, this means that during the following run of the game $F'_q$, player \II{} always respects the rules:

\smallskip

\begin{tabular}{ccccccc}
\textbf{I} & $p'_0$ & & $p'_1$ & & $\hdots$ & \\
\textbf{II} & & $x_0, \varepsilon_0$ & & $x_1, \varepsilon_1$ & & $\hdots$ 
\end{tabular}

\smallskip

\noindent But for every $i \in \omega$, we have that $p_i = \tau(x_0, \ldots, x_{n - 1})$ and $p'_i = \tau'(x_0, \varepsilon_0, \ldots, x_{n - 1}, \varepsilon_{n - 1})$, so by definition of $\tau$, we have $p_i \leqslant p'_i$. Since player \II{} respects the rules in $F_q$, we have that $(x_0, \ldots, x_i) \vartriangleleft p_i$, so $(x_0, \varepsilon_0, \ldots, x_i, \varepsilon_i) \vartriangleleft p'_i$, and \II{} also respects the rules in $F'_q$. This concludes the proof.

\end{proof}

\end{proof}

Observe that in the proof of Theorem \ref{abstractRosendal}, we only need Theorem \ref{main} for $G_\delta$ sets, and hence determinacy for $G_\delta$ games. Hence, unlike Theorem \ref{main} in its generality, the last result is provable in $ZC$. Actually, as previously, for effective Gowers spaces, it is even provable in $Z + DC$.

\smallskip

Again, we actually proved a little more. For a suitable class $\Gamma$ of subsets of Polish spaces, let $\exists\Gamma$ be the class of projections of $\Gamma$-sets; in other words, for $A$ a subset of a Polish space $Y$, we have $A \in \exists\Gamma$ if and only if there exist $B \in Y \times 2^\omega$ such that $B \in \Gamma$ and $A$ is the first projection of $B$ (we could have taken any uncountable Polish space instead of $2^\omega$ in this definition, since $\Gamma$ is closed under Borel inverse images). Then the proof of Theorem \ref{abstractRosendal}, combined with Corollaries \ref{AdvGamma} and \ref{AdvAD}, actually shows the following:

\begin{cor}\label{StratGamma}

\alaligne

\begin{enumerate}

\item Let $\Gamma$ be a suitable class of subsets of Polish spaces. If every $\Gamma$-subset of $\R^\omega$ is determined, then for an analytic Gowers space $\G = (P, X, \leqslant, \leqslant^*, \vartriangleleft)$, every $\exists\Gamma$-subset of $X^\omega$ is strategically Ramsey.

\item $(ZF + DC + AD_\R)$ Let $\G = (P, X, \leqslant, \leqslant^*, \vartriangleleft)$ be an effective Gowers space such that $P$ is a subset of a Polish space. Then every subset of $X^\omega$ is strategically Ramsey.

\end{enumerate}

\end{cor}

In particular, if there exists a measurable cardinal, then in an analytic Gowers space, every $\mathbf{\Sigma}_2^1$-set is strategically Ramsey.

\smallskip

The rest of this section aims to explain how we can, in certain cases, get symmetrical Ramsey results like Mathias--Silver's theorem from Theorem \ref{abstractRosendal}, which is asymmetrical. By \emph{asymmetrical}, we mean here that unlike Mathias--Silver's theorem, in Theorem \ref{abstractRosendal}, both possible conclusion don't have the same form. Actually, one of these conclusions is stronger than the other (and, as it will turn out later, \emph{strictly} stronger in general), as is shown by the following lemma.

\begin{lem}\label{ImplGames}

Let $\X \subseteq X^\omega$ and $p \in P$. Suppose that \I{} has a strategy in $F_p$ to reach $\X$. Then \II{} has a strategy in $G_p$ to reach $\X$.

\end{lem}

\begin{proof}

Fix $\tau$ a strategy enabling \I{} to reach $\X$ in $F_p$. We describe a strategy for \II{} in $G_p$ by simulating a play $(q_0, x_0, q_1, x_1, \ldots)$ of $G_p$ by a play $(p_0, x_0, p_1, x_1, \ldots)$ of $F_p$ having the same outcome and during which \I{} always plays according to $\tau$; this will ensure that $(x_0, x_1, \ldots) \in \X$ and that this play of $G_p$ will be winning for \II{}.

\smallskip

Suppose that the first $n$ turns of both games have been played, which means that the $p_i$'s, the $q_i$'s and the $x_i$'s have been choosen for every $i < n$. For the next turn, in $G_p$, player \I{} plays $q_n \leqslant p$, and in $F_p$, the strategy $\tau$ tells \I{} to play $p_n \lessapprox p$. Then $q_n \leqslant^* p_n$, so by axiom 2. in the definition of a Gowers space, there exists $r_n \in P$ such that $r_n \leqslant p_n$ and $r_n \leqslant q_n$. Let $x_n \in X$ such that $(x_0, \ldots, x_n) \vartriangleleft r_n$ (existing by axiom 4.). Then $x_n$ can be legally played by \II{} in both $F_p$ and $G_p$, what concludes the proof.

\end{proof}

Actually, the fact that \I{} has a strategy in $F_p$ to reach some set $\X$ is in general much stronger than the fact, for \II{}, to have a strategy in $G_p$ to reach the same set, and the first statement is in fact very close to a ``genuine'' Ramsey statement. By a ``genuine'' Ramsey statement, we mean a non-game-theoretical statement of the form ``every sequence $(x_n)_{n \in \omega}$ such that $\forall n \in \omega \; (x_0, \ldots, x_n) \vartriangleleft p$, and moreover satisfiying some structural condition, belongs to $\X$'', like in the conclusions of Mathias--Silver's theorem. In the case of the Mathias--Silver space, the link between the existence of a strategy for \I{} in the asymptotic game and a genuine Ramsey statement is given by the following lemma:

\begin{lem}\label{asympexact}

Work in the Mathias--Silver space, and let $\X \subseteq \omega^\omega$. Suppose that, for some $M \in \elinf$, player \I{} has a strategy in $F_M$ to reach $\X$. Then there exists an infinite $N \subseteq M$ such that every infinite $S\subseteq N$ belongs to  $\X$ (here, we identify infinite subsets of $\omega$ with increasing sequences of integers).

\end{lem}

Obviously, a weak converse of this lemma holds: if every infinite $S \subseteq M$ belongs to $\X$, then \I{} has a strategy in $F_M$ to reach $\X$. Indeed, he can always ensure that the outcome of this game is an increasing sequence.

\begin{proof}[Proof of Lemma \ref{asympexact}]

Without loss of generality, assume $M = \omega$. As in the proof of Fact \ref{Fact34}, consider that in $F_\omega$, player \II{} is allowed to play against the rules, but loses if she does. Let $\tau$ be a strategy for player \I{} in $F_\omega$, enabling him to reach $\X$; in this context, this strategy can be viewed as a mapping associating to each finite sequence of integers a cofinite subset of $\omega$. Without loss of generality, we can assume that these cofinite subsets are final segments of $\omega$; for $s \in \omega^{< \omega}$, let $\tau_0(s) = \min \tau(s)$. Now define, by induction, a strictly increasing sequence $(n_i)_{i \in \omega}$ of integers in the following way: let $n_0 = \tau_0(\varnothing)$, and for $i \in \omega$, let $n_{i + 1}$ be the maximum of $n_i + 1$ and of the $\tau_0(n_{i_0}, \ldots, n_{i_{k - 1}})$'s for $k \in \omega$ and $0 \leqslant i_0 < \ldots < i_{k - 1} = i$. Let $N = \{n_i \mid i \in \omega\}$; then $N$ is as required. Indeed, an infinite subset of $N$ has the form $\{n_{i_k} \mid k \in \omega\}$ for a strictly increasing sequence of integers $(i_k)_{k \in \omega}$. To prove that $(n_{i_k})_{k \in \omega} \in \X$, it is enough to prove this sequence is the outcome of some legal run of the game $F_\omega$ during which player \I{} always plays according to the strategy $\tau$. In other words, letting, for all $k \in \omega$, $P_k = \tau(n_{i_0}, \ldots, n_{i_{k - 1}})$, we have to show that during the following run of the game $F_\omega$, player \II{} always respects the rules:

\smallskip

\begin{tabular}{ccccccc}
\textbf{I} & $P_0$ & & $P_1$ & & $\hdots$ & \\
\textbf{II} & & $n_{i_0}$ & & $n_{i_1}$ & & $\hdots$ 
\end{tabular}

\smallskip

\noindent But by construction, we have that $n_{i_0} \geqslant n_0 = \tau_0(\varnothing) = \min P_0$, and for $k \geqslant 1$, $n_{i_k} \geqslant n_{i_{k - 1} + 1} \geqslant \tau_0(n_{i_0}, \ldots, n_{i_{k - 1}}) = \min P_k$, which concludes the proof.

\end{proof}

The setting of Gowers spaces does not give enough structure to get such a result in general. A general version of this result will be given in section \ref{par5}, in the setting of \emph{approximate asymptotic spaces} with some additional structure; and, in a very different way, the setting of \emph{Ramsey spaces} presented in \cite{todorcevicorange} is also convenient to get non game-theoretical infinite-dimensional Ramsey results.

\smallskip

We now introduce the \emph{pigeonhole principle}, a property of Gowers spaces under which  we can get a Ramsey result involving a strategy for player \I{} in the asymptotic game in both sides. This result is the best possible we can get in this general setting, and Lemma \ref{asympexact} shows that this is the right analogue of Mathias--Silver's theorem.
 
\smallskip
 
In the rest of this paper, denote by $q \subseteq_s A$, for $q \in P$, $s \in X^{< \omega}$ and $A \subseteq X$, the fact that for every $x \in X$ such that $s \concat x \vartriangleleft q$, we have $x \in A$. This notation could sound strange, however, in spaces where $P \subseteq \mathcal{P}(X)$ and $(x_0, \ldots, x_n) \vartriangleleft q \Leftrightarrow x_n \in q$, we have that $q \subseteq_s A$ iff $q \subseteq A$. Let us introduce the pigeonhole principle.

\begin{defin}

The Gowers space $\G$ is said to satisfy the \emph{pigeonhole principle} if for every $p \in P$, $s \in X^{<\omega}$ and $A \subseteq X$, there exists $q \leqslant p$ such that either $q \subseteq_s A$, or $q \subseteq_s A^c$.

\end{defin}

The pigeonhole principle trivially holds in the Mathias--Silver space. It also holds in the Rosendal space over the field $\mathbb{F}_2$: this is actually a rephrasing of Hindman's theorem for $\textup{FIN}$ (see for example \cite{todorcevicorange}, Theorem 2.25). However, it does not hold in the Rosendal space over $K$, for $K \neq \mathbb{F}_2$: to see this, take for example for $A$ the set of all vectors whose first nonzero coordinate is $1$. Note that apart from this trivial obstruction, the pigeonhole principle does not hold in the Rosendal space for much more intrinsic reasons. Indeed, consider the \emph{projective Rosendal space}, i.e. the forgetful Gowers space $\mathcal{PR}_K = (P, \mathbb{P}(E), \subseteq, \subseteq^*, \subseteq)$, where $\mathbb{P}(E)$ is a countably infinite-dimensional projective space over the field $K$ (that is, the set of vector lines of some countably infinite-dimensional $K$-vector space $E$), $P$ is the set of block subspaces of $E$ relative to a fixed basis $(e_i)_{i \in \omega}$ of $E$, $\subseteq^*$ is the inclusion up to finite codimension as in the definition of the Rosendal space, and where since the space is forgetful, the relation usually denoted by $\vartriangleleft$ is viewed as a relation between points and subspaces, here the inclusion. Then for $K \neq \mathbb{F}_2$, the pigeonhole principle still does not hold in $\mathcal{PR}_K$: take for example for $A$ the set of all vector lines $Kx$, where the first and the last non-zero coordinates of $x$ are equal.

\smallskip

Under the pigeonhole principle, we will show a form of converse assertion to Lemma \ref{ImplGames}:

\begin{prop}\label{consPigeonhole}

Suppose that the Gowers space $\G$ satisfies the pigeonhole principle. Let $\X \subseteq X^\omega$ and $p \in P$. If player \II{} has a strategy in $G_p$ to reach $\X$, then there exists $q \leqslant p$ such that \I{} has a strategy in $F_q$ to reach $\X$.

\end{prop}

Before proving this proposition, let us make some remarks. First, Proposition \ref{consPigeonhole} immediately implies the following corollary:

\begin{cor}\label{strongStratRamsey}

Suppose that the Gowers space $\G$ satisfies the pigeonhole principle. Let $\X \subseteq X^\omega$ be a strategically Ramsey set. Then for all $p \in P$, there exists $q \leqslant p$ such that in $F_q$, player \I{} has a strategy either to reach $\X$, or to reach $\X^c$.

\end{cor}

This corollary has some kind of converse. Indeed, for every $s \in X^{< \omega}$, consider the Gowers space $\G^s = (P, X, \leqslant, \leqslant^*, \vartriangleleft^s)$, where $P$, $X$, $\leqslant$ and $\leqslant^*$ are the same as in $\G$ and where $t \vartriangleleft^s p \Leftrightarrow s \concat t \vartriangleleft p$. Then it is not hard to see that $\G$ satisfies the pigeonhole principle if and only if the conclusion of Corollary \ref{strongStratRamsey} holds for all sets of the form $\{(x_n)_{n \in \omega} \mid x_0 \in A\}$ (where $A \subseteq X$), in the space $\G^s$ for every $s$. Indeed, in $\G^s$, player \I{} has a strategy in $F_q$ to reach the set $\{(x_n)_{n \in \omega} \mid x_0 \in A\}$ if and only if there exists $q_0 \lessapprox q$ such that $q_0 \subseteq_s A$. In the particular case of a forgetful space, the satisfaction of the conclusion of the last corollary for clopen sets is thus equivalent to the pigeonhole principle.


\smallskip

Also observe that Corollary \ref{strongStratRamsey} applied to the Mathias--Silver space, combined with Lemma \ref{asympexact}, gives that a set $\X \subseteq \elinf$ is Ramsey (in the sense of Mathias--Silver's theorem) if and only if it is strategically Ramsey in the Mathias--Silver space (when seen as a subset of $\omega^\omega$). In particular, Mathias--Silver's theorem is a consequence of the abstract Rosendal's Theorem \ref{abstractRosendal}.

\smallskip

We now prove Proposition \ref{consPigeonhole}.

\begin{proof}[Proof of Proposition \ref{consPigeonhole}]

Fix $\tau$ a strategy for \II{} in $G_p$ to reach $\X$. As in the proof of Proposition \ref{mainaux}, call a \emph{state} a partial play of $G_p$ either empty or ending with a move of \II{}, during which \II{} always plays according to her strategy. Say that a state \emph{realises} a sequence $(x_0, \ldots, x_{n - 1}) \in X^{< \omega}$ if it has the form $(p_0, x_0, \ldots, p_{n - 1}, x_{n - 1})$. Define in the same way the notion of a \emph{total state} (which is a total play of $G_p$) and of realisation for a total state; if an infinite sequence is realised by some total state, then it belongs to $\X$. Say that a point $x \in X$ is \emph{reachable} from a state $\st$ if there exists $r \leqslant p$ such that $\tau(\st \concat r) = x$. Denote by $\Ast$ the set of all points that are reachable from the state \nolinebreak $\st$. We will use the following fact.

\begin{fact}\label{factdeuxonze}

For every state $\st$ realising a finite sequence $s$, and for every $q \leqslant p$, there exists $r \leqslant q$ such that $r \subseteq_s \Ast$.

\end{fact}

\begin{proof}

Otherwise, by the pigeonhole principle, there would exist $r \leqslant q$ such that $r \subseteq_s (\Ast)^c$. But then \I{} could play $r$ after the partial play $\st$, and \II{} would answer, according to her strategy, by $x = \tau(\st \concat r)$ that should satisfy $s \concat x \vartriangleleft r$. Since $r \subseteq_s (\Ast)^c$, this would imply that $x \in (\Ast)^c$. But we also have, by definition of $\Ast$, that $x \in \Ast$, a contradiction.

\end{proof}

Now let $(s_n)_{n \in \omega}$ be an enumeration of $X^{< \omega}$ such that if $s_m \subseteq s_n$, then $m \leqslant n$. We define, for some $n \in \omega$, a state $\st_n$ realising $s_n$, by induction in the following way: $\st_0 = \varnothing$ and for $n \geqslant 1$, letting $s_n = s_m \concat x$ for some $m < n$ and some $x \in X$,

\begin{itemize}

\item if $\st_m$ has been defined and if $x$ is reachable from $\st_m$, then choose a $r \leqslant p$ such that $x = \tau(\st_m \concat r)$ and put $\st_n = \st_m \concat (r, x)$,

\item otherwise, $\st_n$ is not defined.

\end{itemize}

\noindent Observe that if $\st_n$ is defined and if $s_m \subseteq s_n$, then $\st_m$ is defined and $\st_m \subseteq \st_n$.

\smallskip

We now define a $\leqslant$-decreasing sequence $(q_n)_{n \in \omega}$ of elements of $P$ in the following way: $q_0 = p$ and

\begin{itemize}

\item if $\st_n$ is defined, then $q_{n + 1}$ is the result of the application of Fact \ref{factdeuxonze} to $\st_n$ and $q_n$;

\item $q_{n + 1} = q_n$ otherwise.

\end{itemize}

Finally, let $q \leqslant p$ be such that for every $n \in \omega$, $q \leqslant^* q_n$. We will show that \I{} has a strategy in $F_q$ to reach $\X$. We describe this strategy on the following play of $F_q$:

\smallskip

\begin{tabular}{ccccccc}
\textbf{I} & $u_0$ & & $u_1$ & & $\hdots$ & \\
\textbf{II} & & $x_0$ & & $x_1$ & & $\hdots$ 
\end{tabular}

\smallskip

\noindent We actually show that \I{} can always play preserving the fact that, if $n_i \in \omega$ is such that $s_{n_i} = (x_0, \ldots, x_{i - 1})$, then $\st_{n_i}$ is defined. This will be enough to conclude: indeed, $\bigcup_{i \in \omega} \st_{n_i}$ will be a total state realising the sequence $(x_i)_{i \in \omega}$, showing that this sequence belongs to $\X$.

\smallskip

Suppose that the $i$\textsuperscript{th} turn of the play has just been played, so the sequence $s_{n_i} = (x_0, \ldots, x_{i - 1})$ has been defined, in such a way that $\st_{n_i}$ is defined. Then by construction of $q_{n_i + 1}$, we have that $q_{n_i + 1} \subseteq_{s_{n_i}} \Asti{n_i}$. We let \I{} play some $u_i$ such that $u_i \lessapprox q$ and $u_i \leqslant q_{n_i + 1}$. Then $u_i \subseteq_{s_{n_i}} \Asti{n_i}$, so whatever is the $x_i$ that \II{} answers with, this $x_i$ is reachable from $\st_{n_i}$. So if $s_{n_{i+ 1}} = s_{n_i} \concat x_i$, then $\st_{n_{i + 1}}$ has been defined, and the wanted property is preserved.

\end{proof}

Observe that this proof can be done in $ZF + DC$, even if the space $\G$ is not supposed effective.

\bigskip\bigskip

\section{Approximate Gowers spaces}\label{par4}

For some applications, for instance in Banach-space theory, it would be useful to have results similar to those presented in the previous sections, but for spaces with an uncountable set of points. However, it turns out that in the definition of a Gowers space, this hypothesis is necessary: without it, Theorems \ref{main} and \ref{abstractRosendal} are not true in general. We begin this section with presenting a counterexample for this. Then, following an idea introduced by Gowers for his Ramsey-type Theorem  \ref{GowersThm}, we introduce a metric version of Gowers spaces, allowing us to get approximate Ramsey-type results in situations where the set of points is uncountable. The results of this section, along with these of the next section, will allow us to directly recover results like Gowers' Theorems \ref{GowersThm} and \ref{Gowersc0theorem}. The interest of the spaces we introduce here is more practical that theoretical: their main aim is to allow applications, for instance in Banach-space geometry.

\begin{prop}

There exists a space $\G = (P, X, \leqslant, \leqslant^*, \vartriangleleft)$ satisfying all the axioms of forgetful Gowers spaces, apart from the fact that here, $X$ is not a countable set but an uncountable Polish space, in which not every Borel set is strategically Ramsey, neither every Borel set is adversarially Ramsey.

\end{prop}

\begin{proof}[Sketch of proof]

Work in the $\R$-vector space $\R^\omega$, endowed with the product topology.  For $x = (x^i)_{i \in \omega}\in \R^\omega$, let $\supp(x) = \{i \in \omega \mid x^i \neq 0\}$. Relative to this notion of support, define as usual the notion of a block-sequence, and define a block subspace as a closed subspace spanned by a block-sequence. Let $P$ be the set of block-subspaces, and $X = \R^\omega \setminus \{0\}$. The space $\G$ is defined similarly as the Rosendal space: $\G = (P, X, \subseteq, \subseteq^*, \in)$.

\smallskip

We sketch the construction of a Borel set $\X \subseteq X^\omega$ that is not strategically Ramsey (here, $X^\omega$ is endowed with the product of the Polish topologies on $X$).  Observe that the set $\{(x_n)_{n \in \omega} \in X^\omega \mid (x_0, x_2, \ldots) \in \X\}$ is also Borel and is not adversarially Ramsey.

\smallskip

For $x = (x^i)_{i \in \omega}\in X$, let $N(x) = x^{\min \supp(x)}$. Let $\widetilde{P}$ denote the set of block-sequences; this is a closed subset of $X^\omega$, so there is a Borel isomorphism $\varphi : \R^* \longrightarrow \widetilde{P}$. We define the set $\X$ in the following way: $(x_n)_{n \in \omega}$ is in $\X$ if and only if $x_1$ is equal to a term of the block sequence $\varphi(N(x_0))$. Then $\X$ is not strategically Ramsey. Indeed, given $p \in P$, player \II{} cannot have a strategy in $G_p$ to reach $\X$, because if she played $x_0$ at the first turn, then it is not hard to see that at the second turn, player \I{} can always play a block-subspace $p_1 \subseteq p$ containing no term of the block-sequence $\varphi(N(x_0))$. And player \I{} cannot either have a strategy in $F_p$ to reach $\X^c$, because at the first turn, player \II{} can always play a vector $x_0$ such that $\varphi(N(x_0))$ is a block-sequence spanning the block-subspace $p$. So the subspace $p_1 \lessapprox p$ played by \I{} will necessarily contain a term of the sequence $\varphi(N(x_0))$, and \II{} will be able to play this term to reach her goal.

\end{proof}

\begin{rem}

In the last proof, if we endow $X^\omega$ with the product of the discrete topologies on $X$ instead of the product of the Polish topologies, then the counterexample $\X$ we built is actually clopen.

\end{rem}

We now define the metric version of Gowers spaces we will use for the rest of this paper.

\begin{defin}

An \emph{approximate Gowers space} is a sextuple $\G = (P, X, d, \leqslant, \leqslant^*, \vartriangleleft)$, where $P$ is a nonempty set, $X$ is a nonempty Polish space, $d$ is a compatible distance on $X$, $\leqslant$ and $\leqslant^*$ are two quasiorders on $P$, and $\vartriangleleft \;\subseteq X \times P$ is a binary relation, satisfying the same axioms 1. -- 3. as in the definition of a Gowers' space and satisfying moreover the two following axioms:

\begin{enumerate}

\setcounter{enumi}{3}

\item for every $p \in P$, there exists $x \in X$ such that $x \vartriangleleft p$;

\item for every $x \in X$ and every $p, q \in P$, if $x \vartriangleleft p$ and $p\leqslant q$, then $x \vartriangleleft q$.

\end{enumerate}

\noindent The relation $\lessapprox$ and the compatibility relation on $P$ are defined in the same way as for a Gowers space.

\smallskip

For $p \in P$, define the games $A_p$, $B_p$, $F_p$, and $G_p$ exactly in the same way as for Gowers spaces (see Definitions \ref{DefAdversarialGames} and \ref{DefGowersGames}), except that the rules $(x_0, y_0, \ldots, x_{i - 1}, y_{i - 1}, x_i) \vartriangleleft p_i$ and $(x_0, y_0, \ldots, x_i, y_i) \vartriangleleft q_i$ in the definition of $A_p$ and $B_p$ and the rule $(x_0, \ldots, x_i) \vartriangleleft p_i$ in the definition of $F_p$ and $G_p$,  will be naturally replaced with respectively $x_i \vartriangleleft p_i$, $y_i \vartriangleleft q_i$ and $x_i \vartriangleleft p_i$. The outcome is, there, an element of $X^\omega$.

\end{defin}

Observe that, with this definition, approximate Gowers spaces are always forgetful, that is, the relation $\vartriangleleft$ is defined as a subset of $X \times P$ and not as a subset of $\Seq(X) \times P$.
This technical restriction is needed for our results (in particular in the proof of Theorem \ref{ThmApproxRamsey}).

\smallskip

In the rest of this section, we fix an approximate Gowers space $\G = (P, X, d, \leqslant\nolinebreak, {\leqslant^*}\nolinebreak, {\vartriangleleft} \nolinebreak)$. An important notion in the setting of approximate Gowers spaces is that of \emph{expansion}.

\begin{defin}

\alaligne

\begin{enumerate}

\item Let $A \subseteq X$ and $\delta > 0$. The \emph{$\delta$-expansion} of $A$ is the set $(A)_\delta = \{x \in X \mid \exists y \in \nolinebreak A \; d(x, y) \leqslant \delta\}$;

\item Let $\X \subseteq X^\omega$ and $\Delta = (\Delta_n)_{n \in \omega}$ be a sequence of positive real numbers. The \emph{$\Delta$-expansion} of $\X$ is the set $(\X)_\Delta = \{(x_n)_{n \in \omega} \in X^\omega \mid \exists (y_n)_{n \in \omega} \in \X \; \forall n \in \omega \; d(x_n, y_n) \leqslant \Delta_n\}$.

\end{enumerate}

\end{defin}

We can now define the notions of adversarially Ramsey sets and of strategically Ramsey sets in an approximate Gowers space:

\begin{defin}

Let $\X \subseteq X^\omega$.

\begin{enumerate}

\item We say that $\X$ is \emph{adversarially Ramsey} if for every sequence $\Delta$ of positive real numbers and for every $p \in P$, there exists $q \leqslant p$ such that either player \I{} has a strategy in $A_q$ to reach $(\X)_\Delta$, or player \II{} has a strategy in $B_q$ to reach $(\X^c)_\Delta$.

\item We say that $\X$ is \emph{strategically Ramsey} if for every sequence $\Delta$ of positive real numbers and for every $p \in P$, there exists $q \leqslant p$ such that either player \I{} has a strategy in $F_q$ to reach $\X^c$, or player \II{} has a strategy in $G_q$ to reach $(\X)_\Delta$.

\end{enumerate}

\end{defin}

Observe that if $\G_0 = (P, X, \leqslant, \leqslant^*, \vartriangleleft)$ is a forgetful Gowers space (where we consider $\vartriangleleft$ as a subset of $X \times P$), then we can turn it into an approximate Gowers space $\G_0' = (P, X, d, \leqslant, \leqslant^*, \vartriangleleft)$ by taking for $d$ the discrete distance on $X$ ($d(x, y) = 1$ for $x \neq y$). In this way, for $0 < \delta < 1$ and $A \subseteq X$ we have $(A)_\delta = A$, and for $\Delta$ a sequence of positive real numbers strictly lower than $1$ and for $\X \subseteq X^\omega$, we have $(\X)_\Delta = \X$. So for a set $\X \subseteq X^\omega$, the definition of being adversarially or strategically Ramsey in $\G_0$ and in $\G_0'$ coincide. Therefore, we will consider forgetful Gowers spaces as particular cases of approximate Gowers spaces.

\smallskip

Another interesting family of examples of approximate Gowers spaces is the following. Given a Banach space $E$ with a Schauder basis $(e_i)_{i \in \omega}$, we can consider the \emph{canonical approximate Gowers space over $E$}, $\G_E = (P, S_E, d, \subseteq, \subseteq^*, \in)$, where $P$ is the set of all block subspaces of $E$, $S_E$ is the unit sphere of $E$, $d$ the distance given by the norm, and $X \subseteq^* Y$ if and only if $Y$ contains some finite-codimensional block subspace of $X$. We will see in the next section how to get Gowers' Theorems \ref{GowersThm} and \ref{Gowersc0theorem} from the study of this space.

\smallskip

The results that generalize Theorems \ref{main} and \ref{abstractRosendal} to adversarial Gowers spaces are the following:

\begin{thm}\label{ThmApproxRamsey}

\alaligne

\begin{enumerate}

\item Every Borel subset of $X^\omega$ is adversarially Ramsey;

\item Every analytic subset of $X^\omega$ is strategically Ramsey.

\end{enumerate}

\end{thm}

\begin{proof}

Observe that to prove \emph{2.}, it is actually sufficient to prove the following apparently weaker result: for every $\X \subseteq X^\omega$ analytic, for every sequence $\Delta$ of positive real numbers and for every $p \in P$, there exists $q \leqslant p$ such that either player \I{} has a strategy in $F_q$ to reach $(\X^c)_\Delta$, or player \II{} has a strategy in $G_q$ to reach $(\X)_\Delta$. Indeed, if $\X$ is analytic, then $(\X)_\frac{\Delta}{2}$ is analytic too; so applying the last result to $(\X)_\frac{\Delta}{2}$ and to the sequence $\frac{\Delta}{2}$, and using the fact that $\left(\left((\X)_{\frac{\Delta}{2}}\right)^c\right)_\frac{\Delta}{2} \subseteq \X^c$ and $\left((\X)_{\frac{\Delta}{2}}\right)_{\frac{\Delta}{2}} \subseteq (\X)_{\Delta}$, we get that $\X$ is strategically Ramsey.

\smallskip

Now let $D \subseteq X$ be a countable dense subset, and $\Delta$ be a sequence of positive real numbers. Consider the Gowers space $\G_\Delta = (P, D, \leqslant, \leqslant^*, \vartriangleleft_\Delta)$, where $\vartriangleleft_\Delta$ is defined by $(y_0, \ldots, y_n) \vartriangleleft_\Delta p$ if there exists $x_n \in X$ with $x_n \vartriangleleft p$ and $d(x_n, y_n) < \Delta_n$. To avoid confusion, denote by $A_p$, $B_p$, $F_p$ and $G_p$ the games in the space $\G$, and by $A_p^\Delta$, $B_p^\Delta$, $F_p^\Delta$ and $G_p^\Delta$ the games in the space $\G_\Delta$.

\smallskip

If $\X$ is Borel (resp. analytic) then the set $\X \cap D^\omega$ is Borel (resp. analytic) too (when $D$ is endowed by the discrete topology), so it is adversarially (resp. strategically) Ramsey in $\G_\Delta$. So to prove the theorem, it is enough to show that for every $p \in P$, we have that:

\begin{itemize}

\item[(i)] if player \I{} has a strategy in $F_p^\Delta$ to reach $\X^c$, then he has a strategy in $F_p$ to reach $(\X^c)_\Delta$;

\item[(ii)] if player \II{} has a strategy in $G_p^\Delta$ to reach $\X$, then she has a strategy in $G_p$ to reach $(\X)_\Delta$;

\item[(iii)] if player \I{} has a strategy in $A_p^\Delta$ to reach $\X$, then he has a strategy in $A_p$ to reach $(\X)_\Delta$;

\item[(iv)] if player \II{} has a strategy in $B_p^\Delta$ to reach $\X^c$, then she has a strategy in $B_p$ to reach $(\X^c)_\Delta$.

\end{itemize}

\noindent We only prove (i) and (ii); the proofs of (iii) and (iv) are naturally obtained by combining the proofs of (i) and (ii).

\begin{itemize}

\item[(i)] As usual, we fix a strategy for \I{} in $F_p^\Delta$, enabling him to reach $\X^c$, and we describe a strategy for \I{} in $F_p$ to reach $(\X^c)_\Delta$ by simulating a play $(p_0, x_0, p_1, x_1, \ldots)$ of $F_p$ by a play $(p_0, y_0, p_1, y_1, \ldots)$ of $F_p^\Delta$ in which \I{} always plays using his strategy; we moreover suppose that the same subspaces are played by \I{} in both games.

\smallskip

Suppose that in both games, the first $n$ turns have been played, so the $p_i$'s, the $x_i$'s and the $y_i$'s are defined for $i < n$. According to his strategy, in $F_p^\Delta$, \I{} plays some $p_n \lessapprox p$. Then we let \I{} play the same $p_n$ in $F_p$, and in this game, \II{} answers with $x_n \in X$ such that $x_n \vartriangleleft p_n$. Then we choose $y_n \in D$ such that $d(x_n, y_n) < \Delta_n$; by the definition of $\vartriangleleft_\Delta$, we have that $(y_0, \ldots, y_n) \vartriangleleft_\Delta p_n$, so we can let \II{} play $y_n$ in $F_p^\Delta$, and the games can continue!

\smallskip

Due to the choice of the strategy of \I{} in $F_p^\Delta$, we get that $(y_n)_{n \in \omega} \in \X^c$, so $(x_n)_{n \in \omega} \in (\X^c)_\Delta$ as wanted.

\item[(ii)] We simulate a play $(p_0, x_0, p_1, x_1, \ldots)$ of $G_p$ by a play $(p_0, y_0, p_1, y_1, \ldots)$ of $G_p^\Delta$ where \II{} uses a strategy to reach $\X$, and we moreover suppose that \I{} plays the same subspaces in both games. Suppose that the first $n$ turns of boths games have been played. In $G_p$, \I{} plays $p_n$. We make \I{} copy this move in $G_p^\Delta$, and according to her strategy, \II{} answers, in this game, by a $y_n \in D$ such that $(y_0, \ldots, y_n) \vartriangleleft_\Delta p_n$. We can find $x_n \in X$ such that $x_n \vartriangleleft p_n$ and $d(x_n, y_n) < \Delta_n$; we let \II{} play this $x_n$ in $G_p$ and the games continue. At the end, we have that $(y_n)_{n \in \omega} \in \X$, so $(x_n)_{n \in \omega} \in (\X)_\Delta$ as wanted.

\end{itemize}

\end{proof}

Say that the approximate Gowers space $\G$ is \emph{analytic} if $P$ is an analytic subset of a Polish space, if the relation $\leqslant$ is a Borel subset of $P^2$, and if for every open set $U \subseteq X$, the set $\{p \in P \mid \exists x \in U \; x \vartriangleleft p\}$ is a Borel subset of $P$. Also recall that if $Y$ is a Polish space, and if $\mathcal{F}(Y)$ is the set of all closed subsets of $Y$, the \emph{Effros Borel structure} on $\mathcal{F}(Y)$ is the $\sigma$-algebra generated by the sets $\{F \in \mathcal{F}(Y) \mid F \cap U \neq \varnothing\}$ where $U$ varies over open subsets of $Y$; with this $\sigma$-algebra, $\mathcal{F}(Y)$ is a standard Borel space (see for example \cite{kechris}, Theorem 12.6). If $P$ is an analytic subset of $\mathcal{F}(X)$ endowed with the Effros Borel structure, and if $\subseteq$ and $\vartriangleleft$ are respectively the inclusion and the membership relation, then $\G$ is an analytic approximate Gowers space. This is, for instance, the case of the canonical approximate Gowers space $\G_E$ over a Banach space $E$ with a basis: indeed, the fact that $F \in \mathcal{F}(S_E)$ is the unit sphere of a block subspace of $E$ can be written ``there exists a block sequence $(x_i)_{i \in \omega}$ such that for every $U$ in a countable basis of open subsets of $S_E$, $F \cap U \neq \varnothing$ if and only if there exists $n \in \omega$ and $(a_i)_{i < n} \in \mathbb{Q}^n \setminus \{0\}$ with $\frac{\sum_{i < n}a_i x_i}{\left\|\sum_{i < n}a_i x_i\right\|} \in U$''.

\smallskip

Observe that if $\G$ is an analytic approximate Gowers space and $\Delta$ a sequence of positive real numbers, then the Gowers space $\G_\Delta$ defined in the proof of Theorem \ref{ThmApproxRamsey} is analytic. So this proof, combined with Corollaries \ref{AdvGamma} and \ref{StratGamma}, gives us the following:

\begin{cor}\label{ApproxGamma}

Let $\Gamma$ be a suitable class of subsets of Polish spaces. Suppose that every $\Gamma$-subset of $\R^\omega$ is determined. Then for every analytic approximate Gowers space $\G = (P, X, d, \leqslant, \leqslant^*, \vartriangleleft)$, we have that:

\begin{enumerate}

\item every $\Gamma$-subset of $X^\omega$ is adversarially Ramsey;

\item every $\exists\Gamma$-subset of $X^\omega$ is strategically Ramsey.

\end{enumerate}

\end{cor}

However, it is not straightforward, in the setting of approximate Gowers spaces, to get results in $ZF + DC + AD_\R$, because the proof of Theorem \ref{ThmApproxRamsey} uses the full axiom of choice. Indeed, since there is, in general, an uncountable number of subspaces, in the proof of (ii) (and the same will happen in the proofs of (iii) and (iv)), player \II{} needs $AC$ to choose $x_n$ such that $d(x_n, y_n) < \Delta_n$ and $x_n \vartriangleleft p_n$. However, under a slight restriction, we can get a positive result. Define the notion of an \emph{effective approximate Gowers space} exactly in the same way as for effective Gowers spaces. Effective forgetful Gowers spaces are obviously effective when seen as approximate Gowers spaces, but also, the canonical approximate Gowers space $\G_E$ is effective (this can be shown in the same way as for the Rosendal space). If $\G$ is an effective approximate Gowers space and $\Delta$ a sequence of positive real numbers, then the Gowers space $\G_\Delta$ defined in the proof of Theorem \ref{ThmApproxRamsey} is also effective. And we have:

\begin{cor}[$ZF + DC + AD_\R$]\label{ApproxAD}

Let $\G = (P, X, d, \leqslant, \leqslant^*, \vartriangleleft)$ be an effective approximate Gowers space such that $P$ is a subset of a Polish space, and such that for every $p \in P$, the set $\{x \in X \mid x \vartriangleleft p\}$ is closed in $X$. Then every subset of $X^\omega$ is adversarially Ramsey and strategically Ramsey.

\end{cor}

\begin{proof}

We follow the proof of Theorem \ref{ThmApproxRamsey}, using Corollaries \ref{AdvAD} and \ref{StratGamma} to get that the set $\X \cap D^\omega$ is adversarially Ramsey and strategically Ramsey in $\G_\Delta$. The only thing to do is to verify that the proofs of (i)--(iv) can be carried out with only $DC$ instead of $AC$; as previously, we only do it for (i) and (ii). In the proof of (i), we have to be able to choose $y_n \in D$ such that $d(x_n, y_n) < \Delta_n$; this can be done by fixing, at the beginning of the proof, a well-ordering of $D$, and by choosing, each time, the least such $y_n$. In the proof of (ii), the difficulty is to choose $x_n$; so we have to prove that given $p \in P$, $n \in \omega$, and $y \in D$, if there exists $x \in X$ with $x \vartriangleleft p$ and $d(x, y) < \Delta_n$, then we are able to choose such an $x$ without using $AC$.

\smallskip

Using countable choices, for every $y \in D$ and $n \in \omega$, we choose $f_{y, n} : \omega^\omega \longrightarrow B(y, \Delta_n)$ a continuous surjection. Given $p$, $n$ and $y$ as in the previous paragraph, we can let $F=\{u \in \omega^\omega \mid f_{y, n}(u) \vartriangleleft p\}$, a closed subset of $\omega^\omega$. Consider $T \subseteq \omega^{< \omega}$ the unique pruned tree such that $F = [T]$. Then we can let $u$ be the leftmost branch of $T$ and let $x = f_{y, n}(u)$.

\end{proof}

\smallskip

We now introduce the pigeonhole principle in an approximate Gowers space and its consequences. We actually only need an approximate pigeonhole principle in this setting. For $q \in P$ and $A \subseteq X$, we abusively write $q \subseteq A$ to say that $\forall x \in X \; (x \vartriangleleft q \Rightarrow x \in A)$.

\begin{defin}

The approximate Gowers space $\G$ is said to satisfy the \emph{pigeonhole principle} if for every $p \in P$, $A \subseteq X$, and $\delta > 0$ there exists $q \leqslant p$ such that either $q \subseteq A^c$, or $q \subseteq (A)_\delta$.

\end{defin}

For example, by Theorem \ref{pigeonc0}, the canonical approximate Gowers space $\G_E$ satisfies the pigeonhole principle if and only if $E$ is $c_0$-saturated.

\smallskip

As for Gowers spaces, we have the following proposition:

\begin{prop}\label{ApproxConsPigeonhole}

Suppose that the approximate Gowers space $\G$ satisfies the pigeonhole principle. Let $\X \subseteq X^\omega$, $p \in P$ and $\Delta$ be a sequence of positive real numbers. If player \II{} has a strategy in $G_p$ to reach $\X$, then there exists $q \leqslant p$ such that player \I{} has a strategy in $F_q$ to reach $(\X)_\Delta$.

\end{prop}

Before proving this proposition, let us make some remarks. Using again the fact that $\left((\X)_\frac{\Delta}{2}\right)_\frac{\Delta}{2} \subseteq (\X)_\Delta$, we deduce from Proposition \ref{ApproxConsPigeonhole} the following corollary:

\begin{cor}\label{cor39}

Suppose that the approximate Gowers space $\G$ satisfies the pigeonhole principle. Let $\X \subseteq X^\omega$ be a strategically Ramsey set. Then for every $p \in P$ and every sequence $\Delta$ of positive real numbers, there exists $q \leqslant p$ such that in $F_q$, player \I{} either has a strategy to reach $\X^c$, or has a strategy to reach $(\X)_\Delta$.

\end{cor}

Conversely, it is not hard to see that if the conclusion of Corollary \ref{cor39} holds for all sets of the form $\{(x_n)_{n \in \omega} \in X^\omega \mid x_0 \in F\}$, where $F \subseteq X$ is closed, then the space $\G$ satisfies the pigeonhole principle.

\smallskip

Also observe that if $\G_0$ is a forgetful Gowers space, and if $\G_0'$ is the associated approximate Gowers space, then the pigeonhole principle in $\G_0$ is equivalent to the pigeonhole principle in $\G_0'$, and Proposition \ref{ApproxConsPigeonhole} and Corollary \ref{cor39} are respectively the same as Proposition \ref{consPigeonhole} and Corollary \ref{strongStratRamsey}.

\smallskip

We now prove Proposition \ref{ApproxConsPigeonhole}.

\begin{proof}[Proof of Proposition \ref{ApproxConsPigeonhole}]

Unlike the previous results about approximate Gowers spaces, here we cannot deduce this result from its exact version; thus, we adapt the proof of Proposition \ref{consPigeonhole}. To save notation, we show that there exists $q \leqslant p$ such that \I{} has a strategy in $F_q$ to reach $(\X)_{3\Delta}$.

\smallskip

We fix $\tau$ a strategy for \II{} in $G_p$ to reach $\X$. As usual, call a \emph{state} a partial play of $G_p$ either empty or ending with a move of \II{}, during which \II{} always plays according to her strategy. Say that a state \emph{realises} a sequence $(x_0, \ldots, x_{n - 1}) \in X^{< \omega}$ if it has the form $(p_0, x_0, \ldots, p_{n - 1}, x_{n - 1})$. The \emph{length} of the state $\st$, denoted by $|\st|$, is the length of the sequence it realises. Define in the same way the notion of a \emph{total state} (which is a total play of $G_p$) and of realisation for a total state; if an infinite sequence is realised by a total state, then it belongs to $\X$. Say that a point $x \in X$ is \emph{reachable} from a state $\st$ if there exists $r \leqslant p$ such that $\tau(\st \concat r) = x$. Denote by $\Ast$ the set of all points that are reachable from the state \nolinebreak $\st$. We will use the following fact.

\begin{fact}\label{factdeuxonzebis}

For every state $\st$ and for every $q \leqslant p$, there exists $r \leqslant q$ such that $r \subseteq (\Ast)_{\Delta_{|\stp|}}$.

\end{fact}

\begin{proof}

Otherwise, by the pigeonhole principle, there would exist $r \leqslant q$ such that $r \subseteq (\Ast)^c$. But then \I{} could play $r$ after the partial play $\st$, and \II{} would answer, according to her strategy, by $x = \tau(\st \concat r)$ that should satisfy $x \vartriangleleft r$. Since $r \subseteq (\Ast)^c$, this would imply that $x \in (\Ast)^c$. But we also have, by the definition of $\Ast$, that $x \in \Ast$, a contradiction.

\end{proof}

For two sequences $s, t \in X^{\leqslant \omega}$ of the same length, denote by $d(s, t) \leqslant \Delta$ the fact that for every $i < |s|$, we have $d(s_i, t_i) \leqslant \Delta_i$. Let $D \subseteq X$ be a countable dense set and let $(s_n)_{n \in \omega}$ be an enumeration of $D^{< \omega}$ such that if $s_m \subseteq s_n$, then $m \leqslant n$. We define, for some $n \in \omega$, a state $\st_n$ realising a sequence $t_n$ satisfying $d(s_n, t_n) \leqslant 2\Delta$, by induction in the following way: $\st_0 = \varnothing$ and for $n \geqslant 1$, letting $s_n = s_m \concat y$ for some $m < n$ and some $y \in X$,

\begin{samepage}
\begin{itemize}

\item if $\st_m$ has been defined and if there exists $z \in X$ reachable from $\st_m$ such that $d(y, z) \leqslant 2\Delta_{|\stp_m|}$, then choose a $r \leqslant p$ such that $z = \tau(\st_m \concat r)$ and put $t_n = t_m \concat z$ and $\st_n = \st_m \concat (r, z)$,

\item otherwise, $\st_n$ is not defined.

\end{itemize}
\end{samepage}

\noindent Observe that if $\st_n$ is defined and if $s_m \subseteq s_n$, then $\st_m$ is defined, and we have $\st_m \subseteq \st_n$ and $t_m \subseteq t_n$.

\smallskip

We now define a $\leqslant$-decreasing sequence $(q_n)_{n \in \omega}$ of elements of $P$ in the following way: $q_0 = p$ and

\begin{itemize}

\item if $\st_n$ is defined, then $q_{n + 1}$ is the result of the application of Fact \ref{factdeuxonzebis} to $\st_n$ and $q_n$;

\item $q_{n + 1} = q_n$ otherwise.

\end{itemize}

Finally, let $q \leqslant p$ be such that for every $n \in \omega$, $q \leqslant^* q_n$. We will show that \I{} has a strategy in $F_q$ to reach $(\X)_{3\Delta}$. We describe this strategy on the following play of $F_q$:

\smallskip

\begin{tabular}{ccccccc}
\textbf{I} & $u_0$ & & $u_1$ & & $\hdots$ & \\
\textbf{II} & & $x_0$ & & $x_1$ & & $\hdots$ 
\end{tabular}

\smallskip

\noindent We moreover suppose that at the same time as this game is played, we build a sequence $(n_i)_{i \in \omega}$ of integers, with $n_0 = 0$ and $n_i$ being defined during the $i$\textsuperscript{th} turn, such that $(s_{n_i})_{i \in \omega}$ is increasing and for every $i \in \omega$, $|s_{n_i}| = i$, $\st_{n_i}$ is defined, and $d(s_{n_i}, (x_0, \ldots, x_{i - 1})) \leqslant \Delta$. This will be enough to conclude: indeed, $\bigcup_{i \in \omega} \st_{n_i}$ will be a total state realising the sequence $\bigcup_{i \in \omega} t_{n_i}$, showing that this sequence belongs to $\X$; and since $d\left(\bigcup_{i \in \omega} t_{n_i}, (x_i)_{i \in \omega}\right) \leqslant d\left(\bigcup_{i \in \omega} t_{n_i}, \bigcup_{i \in \omega} s_{n_i}\right) + d\left(\bigcup_{i \in \omega} s_{n_i}, (x_i)_{i \in \omega}\right) \leqslant 3 \Delta$, we will have that $(x_i)_{i \in \omega} \in (\X)_{3\Delta}$.

\smallskip

Suppose that the $i$\textsuperscript{th} turn of the game has just been played, so the sequence $(x_0, \ldots, x_{i - 1})$ and the integers $n_0, \ldots, n_i$ has been defined. Then by construction of $q_{n_i + 1}$, we have that $q_{n_i + 1} \subseteq (\Asti{n_i})_{\Delta_{|\stp_{n_i}|}}$. We let \I{} play some $u_i$ such that $u_i \lessapprox q$ and $u_i \leqslant q_{n_i + 1}$. Then $u_i \subseteq (\Asti{n_i})_{\Delta_{|\stp_{n_i}|}}$. Now, suppose that \II{} answers by $x_i$. Then we choose a $y_i \in D$ such that $d(x_i, y_i) \leqslant \Delta_i$ and we choose $n_{i + 1}$ in such a way that $s_{n_{i + 1}} = s_{n_i} \concat y_i$. So we have that $y_i \in (\Asti{n_i})_{2\Delta_{|\stp_{n_i}|}}$; this shows that $\st_{n_{i + 1}}$ has been defined. Moreover we have $d(s_{n_{i + 1}}, (x_0, \ldots, x_i)) \leqslant \Delta$ as wanted, what ends the proof.

\end{proof}

Again, this proof can be done in $ZF + DC$, even if the space $\G$ is not supposed effective.

\bigskip\bigskip

\section{Eliminating the asymptotic game}\label{par5}

In this section, we provide a tool to deduce, from a statement of the form ``player \I{} has a strategy in $F_p$ to reach $\X$'', a conclusion of the form ``in some subspace, every sequence satisfying some structural condition is in $\X$''. This tool can be seen as a generalization of Lemma \ref{asympexact}. It will allow us to get, from Ramsey results with game-theoretical conclusions, stronger results having the same form as Mathias--Silver's theorem or Gowers' theorem.

\smallskip

We will actually not add any structure on the set of points, but rather provide a tool enabling, in each concrete situation, to build this structure in the way we want. Our result could be stated in the setting of approximate Gowers spaces, but we prefer to state it in the more general setting of \emph{approximate asymptotic spaces}, since it could be useful in itself in situations where we have no natural Gowers space structure.

\smallskip

\begin{defin}

An \emph{approximate asymptotic space} is a quintuple $\A = \{P, X, d, \lessapprox, \vartriangleleft\}$, where $P$ is a nonempty set, $(X, d)$ is a nonempty separable metric space, $\lessapprox$ is a quasiorder on $P$, and $\vartriangleleft \; \subseteq X \times P$ is a binary relation, satisfying the following properties:

\begin{enumerate}

\item for every $p, q, r \in P$, if $q \lessapprox p$ and $r \lessapprox p$, then there exists $u \in P$ such that $u \lessapprox q$ and $u \lessapprox r$;

\item for every $p \in P$, there exists $x \in X$ such that $x \vartriangleleft p$;

\item for every every $x \in X$ and every $p, q \in P$, if $x \vartriangleleft p$ and $p \lessapprox q$, then $x \vartriangleleft q$.

\end{enumerate}

\end{defin}

Every approximate Gowers space has a natural structure of approximate asymptotic space. In an approximate asymptotic space, define the notion of expansion, and the asymptotic game, in the same way as in an approximate Gowers space.

\smallskip

In the rest of this section, we fix $\A = \{P, X, d, \lessapprox, \vartriangleleft\}$ an approximate asymptotic space. Recall that a subset of $X$ is said to be \emph{precompact} if its closure in $X$ is compact. In what follows, for $K \subseteq X$ and $p \in P$, we abusively write $K \vartriangleleft p$ to say that the set $\{x \in K \mid x \vartriangleleft p\}$ is dense in $K$.

\begin{defin}

A \emph{system of precompact sets} for $\A$ is a set $\mathcal{K}$ of precompact subsets of $X$, equipped with an associative binary operation $\oplus$, satisfying the following property: for every $p \in P$, and for every $K, L \in \mathcal{K}$, if $K \vartriangleleft p$ and $L \vartriangleleft p$, then $K \oplus L \vartriangleleft p$.

\smallskip

If $(\K, \oplus)$ is a system of precompact sets for $\A$ and if $(K_n)_{n \in \omega}$ is a sequence of elements of $\K$, then:

\begin{itemize}

\item for $A \subseteq \omega$ finite, denote by $\bigoplus_{n \in A} K_n$ the sum $K_{n_1} \oplus \ldots \oplus K_{n_k}$, where $n_1, \ldots, n_k$ are the elements of $A$ taken in increasing order;

\item a \emph{block sequence} of $(K_n)$ is,  by definition, a sequence $(x_i)_{i \in \omega} \in X^\omega$ for which there exists an increasing sequence of nonempty sets of integers $A_0 < A_1 < A_2 < \ldots$ such that for every $i \in \omega$, we have $x_i \in \bigoplus_{n \in A_i} K_n$.

\end{itemize}

\noindent Denote by $\bs((K_n)_{n \in \omega})$ the set of all block sequences of $(K_n)$.

\end{defin}

We can already give some examples. For the Mathias--Silver space $\Sil$, let $\K_\Sil$ be the set of all singletons, and define the operation $\oplus_\Sil$ by $\{m\} \oplus_\Sil \{n\} = \{max(m, n)\}$. Then $(\K_\Sil, \oplus_\Sil)$ is a system of precompact sets. If $(m_i)_{i \in \omega}$ is an increasing sequence of integers, then the block sequences of  $(\{m_i\})_{i \in \omega}$ are exactly the subsequences of $(m_i)$.

\smallskip

Now, for a Banach space $E$ with a basis, consider the canonical approximate Gowers space $\G_E$. Let $\K_E$ be the set of all unit spheres of finite-dimensional subspaces of $E$. Define the operation $\oplus_E$ on $\K_E$ by $S_F \oplus_E S_G = S_{F + G}$. Then $(\K_E, \oplus_E)$ is a system of precompact sets for $\G_E$. If $(x_n)_{n \in \omega}$ is a (normalized) block sequence of $E$, then for every $n$, $S_{\R x_n} = \{x_n, -x_n\}$ is in $\K_E$, and the block sequences of $(S_{\R x_n})_{n \in \omega}$ in the sense of $\K$ are exactly the (normalized) block sequences of $(x_n)$ in the Banach-theoretical sense. More generally, it is often useful to study the block sequences of sequences of the form $(S_{F_n})_{n \in \omega}$, where $(F_n)_{n \in \omega}$ is a FDD of a closed, infinite-dimensional subspace $F$ of $E$ (that is, a sequence such that every $x \in F$ can be written in a unique way as a sum $\sum_{n = 0}^\infty x_n$, where for every $n$, $x_n \in F_n$).

\smallskip

In general, in an asymptotic space, a sequence $(K_n)_{n \in \omega}$ of elements of a system of precompact sets can be seen as another kind of subspace. Sometimes, some subspaces of the type $(K_n)_{n \in \omega}$ can be represented as elements of $P$; that is, for example, the case in the Mathias--Silver space and in the canonical approximate Gowers space over a Banach space with a basis, as we just saw. We now introduce a theorem enabling us to build sequences $(K_n)_{n \in \omega}$ such that $\bs((K_n)_{n \in \omega}) \subseteq \X$, knowing that player \I{} has a strategy in an asymptotic game to reach $\X$. First, we have to define a new game.

\begin{defin}

Let $(\K, \oplus)$ be a system of precompact sets for the space $\mathcal{A}$, and $p \in P$. The \emph{strong asymptotic game below $p$}, denoted by $SF_p$, is defined as follows:

\smallskip

\begin{tabular}{ccccccc}
\textbf{I} & $p_0$ & & $p_1$ & & $\hdots$ & \\
\textbf{II} & & $K_0$ & & $K_1$ & & $\hdots$ 
\end{tabular}

\smallskip

\noindent where the $K_n$'s are elements of $\K$, and the $p_n$'s are elements of $P$. The rules are the following:

\begin{itemize}

\item for \textbf{I}: for all $n \in \omega$, $p_n \lessapprox p$;

\item for \textbf{II}: for all $n \in \omega$, $K_n \vartriangleleft p_n$.

\end{itemize}

\noindent The outcome of the game is the sequence $(K_n)_{n \in \omega} \in \K^\omega$.

\end{defin}

\begin{thm}\label{AbstractApproxExact}

Let $(\K, \oplus)$ be a system of precompact sets on the space $\mathcal{A}$, $p \in P$, $\X \subseteq X^\omega$, and $\Delta$ be a sequence of positive real numbers. Suppose that player \I{} has a strategy in $F_p$ to reach $\X$. Then he has a strategy in $SF_p$ to build a sequence $(K_n)_{n \in \omega}$ such that $\bs((K_n)_{n \in \omega}) \subseteq (\X)_\Delta$.

\end{thm}

\begin{proof}

For each $K \in \K$, each $q \in P$ such that $K \vartriangleleft q$, and each $i \in \omega$, let $\mathfrak{N}_{i, q}(K)$ be a $\Delta_i$-net in $K$ (that is, a finite subset of $K$ such that $K \subseteq (\mathfrak{N}_{i, q}(K))_{\Delta_i}$), such that for every $x \in \mathfrak{N}_{i, q}(K)$, we have $x \vartriangleleft q$. We fix $\tau$ a strategy for \I{} in $F_p$, enabling him to reach $\X$. As in the proofs of Fact \ref{Fact34} and Lemma \ref{asympexact}, we consider that in $F_p$, \II{} is allowed to play against the rules, but that she immedately loses if she does; so we will view $\tau$ as a mapping from $X^{< \omega}$ to $P$, such that for every $s \in X^{< \omega}$, we have $\tau(s) \lessapprox p$.

\smallskip

Let us describe a strategy for \I{} in $SF_p$ on a play $(p_0, K_0, p_1, K_1, \ldots)$ of this game. Suppose that the first $n$ turns have been played, so the $p_j$'s and the $K_j$'s, for $j < n$, are defined. Moreover suppose that the sequence $(p_j)_{j < n}$ is $\lessapprox$-decreasing. Let $S_{(K_0, \ldots, K_{n - 1})} \subseteq X^{< \omega}$ be the set of all finite sequences $(y_0, \ldots, y_{k - 1})$ satisfying the following property: there exists an increasing sequence $A_0 < \ldots < A_{k - 1}$ of nonempty subsets of $n$ such that for every $i < k$, we have $y_i \in \mathfrak{N}_{i, p_{\min(A_i)}}(\oplus_{j \in A_i} K_j)$. Then $S_{(K_0, \ldots, K_{n - 1})}$ is finite and for every $s \in S_{(K_0, \ldots, K_{n - 1})}$, we have $\tau(s) \lessapprox p$, so by iterating the axiom 1. in the definition of an approximate asymptotic space, we can find $p_n \lessapprox p$ such that for every $s \in S_{(K_0, \ldots, K_{n - 1})}$, we have $p_n \lessapprox \tau(s)$. Moreover, if $n \geqslant 1$, we can choose $p_n$ such that $p_n \lessapprox p_{n - 1}$. The strategy of \I{} will consist in playing this $p_n$.

\smallskip

Now suppose that this play has been played completely; we show that $\bs((K_n)_{n \in \omega}) \subseteq (\X)_\Delta$. Let $(x_i)_{i \in \omega}$ be a block sequence of $(K_n)$ and $A_0 < A_1 < \ldots$ be a sequence of nonempty subsets of $\omega$ such that for every $i$, we have $x_i \in \bigoplus_{n \in A_i} K_n$. For every $i \in \omega$, we have $\left(\bigoplus_{n \in A_i} K_n\right) \vartriangleleft p_{\min(A_i)}$, so $\mathfrak{N}_{i, p_{\min(A_i)}}\left(\bigoplus_{n \in A_i} K_n\right)$ has been defined and we can choose a $y_i$ in it such that $d(x_i, y_i) \leqslant \Delta_i$. We have to show that $(x_i)_{i \in \omega} \in (\X)_\Delta$, so it is enough to show that $(y_i)_{i \in \omega} \in \X$. Knowing that $\tau$ is a strategy for \I{} in $F_p$ to reach $\X$, it is enough to show that, letting $q_i = \tau(y_0, \ldots, y_{i - 1})$ for all $i$, in the following play of $F_p$, \II{} always respects the rules:

\smallskip

\begin{tabular}{ccccccc}
\textbf{I} & $q_0$ & & $q_1$ & & $\hdots$ & \\
\textbf{II} & & $y_0$ & & $y_1$ & & $\hdots$ 
\end{tabular}

\smallskip

\noindent In other words, we have to show that for all $k \in \omega$, we have $y_k \vartriangleleft q_k$.

\smallskip

So let $k \in \omega$. We let $n_0 = \min A_k$. Since the sets $A_0, \ldots, A_{k - 1}$ \nolinebreak are subsets of $n_0$, we have that $(y_0, \ldots, y_{k - 1}) \in S_{(K_0, \ldots, K_{n_0 - 1})}$, and therefore $p_{n_0} \lessapprox \tau(y_0, \ldots, y_{k - 1}) = q_k$.
But $y_k \in \mathfrak{N}_{k, p_{n_0}}\left(\bigoplus_{n \in A_k} K_n\right)$, so $y_k \vartriangleleft p_{n_0}$, so $y_k \vartriangleleft q_k$, as wanted.
\end{proof}


Again, under slight restrictions, we can prove Theorem \ref{AbstractApproxExact} without using the full axiom of choice. Say that the approximate asymptotic space $\mathcal{A}$ is \emph{effective} if 
there exist a function $f : P^2 \longrightarrow P$ such that for every $q, r \in P$, if there exist $p \in P$ such that $q \lessapprox p$ and $r \lessapprox p$, then we have $f(q, r) \lessapprox q$ and $f(q, r) \lessapprox r$. Effective approximate Gowers spaces, when seen as approximate asymptotic spaces, are effective. We will show that if $\A$ is an effective approximate asymptotic space, if $X$ is an analytic subset of a Polish space, if for every $p \in P$, the set $\{x \in X \mid x \vartriangleleft p\}$ is closed in $X$, and if every element of $\K$ is compact, then Theorem \ref{AbstractApproxExact} for $\mathcal{A}$ and $\K$ can be shown in $ZF + DC$. In the proof of Theorem \ref{AbstractApproxExact}, $AC$ is only used: 

\begin{itemize}

\item to choose $p_n$ such that for every $s \in S_{(K_0, \ldots, K_{n - 1})}$, we have $p_n \lessapprox \tau(s)$, and such that $p_n \lessapprox p_{n - 1}$ if $n \geqslant 1$;

\item to choose the nets $\mathfrak{N}_{i, q}(K)$;

\item and to choose $y_i \in \mathfrak{N}_{i, p_{\min(A_i)}}\left(\bigoplus_{n \in A_i} K_n\right)$ such that $d(x_i, y_i) \leqslant \Delta_i$.

\end{itemize}

The choice of the $p_n$'s can be done without $AC$ as soon as the space $\mathcal{A}$ is effective. For the choice of the nets and of the $y_i$'s, first, observe that, given $K \in \K$ and $q \in P$, since $\{x \in X \mid x \vartriangleleft q\}$ is closed in $X$, we have that $K \vartriangleleft q$ if and only if $K \subseteq \{x \in X \mid x \vartriangleleft q\}$; so $\mathfrak{N}_{i, q}(K)$ can actually be an arbitrary $\Delta_i$-net in $K$, and does not need to depend on $q$. Thus, to be able to chose these nets and the $y_i$'s without $AC$, it is enough to show that we can choose, without $AC$, a $\Delta_i$-net $\mathfrak{N}_i(K)$ in $K$ and a wellordering $\prec_{i, K}$ on it, for every $K \in \K$ and every $i \in \omega$. This can be done in the following way. Let $\varphi : \omega^\omega \rightarrow X$ be a continuous surjection. If $K \in \K$, then $\varphi^{- 1}(K)$ has the form $[T_K]$, where $T_K$ is a pruned tree on $\omega$. We can easily build, without choice, a countable dense subset of $[T_K]$, for example the set of all the $u_s$'s where for every $s \in T_K$, $u_s$ is the leftmost branch of $T_K$ satisfying $s \subseteq u_s$. Since $T_K$ can naturally be wellordered, then this dense subset can also be wellordered. Pushing forward by $\varphi$, this enables us to get, for every $K \in \K$, a countable dense subset $D_K \subseteq K$ with a wellordering $\prec_K$. From this we can naturally wellorder the set of all finite subsets of $D_K$, take for $\mathfrak{N}_i(K)$ the least finite subset of $D_K$ that is a $\Delta_i$-net in $K$ and take for $\prec_{i, K}$ the restriction of $\prec_K$ to $\mathfrak{N}_i(K)$.



\smallskip

Theorem \ref{AbstractApproxExact}, combined with the results of the last section and with the last remark, gives us the following corollary:

\begin{cor}[Abstract Gowers' theorem]\label{AbstractGowers}

Let $\G = (P, X, d, \leqslant, \leqslant^*, \vartriangleleft)$ be an approximate Gowers space, equipped with a system of precompact sets $(\K, \oplus)$. Let $\X \subseteq X^\omega$, and suppose that we are in one of the following situations:

\begin{itemize}

\item we work in $ZFC$, and $\X$ is analytic;

\item we work in $ZFC$, $\G$ is analytic and $\X$ is $\exists \Gamma$, for some suitable class $\Gamma$ of subsets of Polish spaces such that every $\Gamma$-subset of $\R^\omega$ is determined;

\item we work in $ZF + DC + AD_\R$, the space $\G$ is effective, $P$ is a subset of Polish space, for every $p \in P$, the set $\{x \in X \mid x \vartriangleleft p\}$ is closed in $X$, and every element of $\K$ is compact.

\end{itemize}

\noindent Let $p \in P$ and $\Delta$ be a sequence of positive real numbers. Then there exists $q \leqslant p$ such that:

\begin{itemize}

\item either player \I{} has a strategy in $SF_q$ to build a sequence $(K_n)_{n \in \omega}$ such that $\bs((K_n)_{n \in \omega}) \subseteq \X^c$;

\item or player \II{} has a strategy in $G_q$ to reach $(\X)_\Delta$.

\end{itemize}

\noindent Moreover, if $\G$ satisfies the pigeonhole principle, then the second conclusion can be replaced with the following stronger one: player \I{} has a strategy in $SF_q$ to build a sequence $(K_n)_{n \in \omega}$ such that $\bs((K_n)_{n \in \omega}) \subseteq (\X)_\Delta$.

\end{cor}

Now see how to deduce Mathias--Silver's theorem, Gowers' Theorem \ref{GowersThm}, and Gowers' theorem for $c_0$ (Theorem \ref{Gowersc0theorem}) from Corollary \ref{AbstractGowers}.

\begin{itemize}

\item For Mathias--Silver's theorem, work in the Mathias--Silver space $\Sil$ with the system $(\K_\Sil, \oplus_\Sil)$ of precompact sets introduced before. Let $M$ be an infinite set of integers, and $\X \subseteq \elinf$ be analytic, that we will consider as a subset of $\omega^\omega$ by identifying infinite subsets of $\omega$ with increasing sequences of integers. Applying Corollary \ref{AbstractGowers} to $\X$, to $M$, and to the constant sequence equal to $\frac{1}{2}$, we get an infinite $N \subseteq M$ such that either \I{} has a strategy in $SF_N$ to build $(\{n_i\})_{i \in \omega}$ with $\bs((\{n_i\})_{i \in \omega}) \subseteq \X$, or he has one to build $(\{n_i\})_{i \in \omega}$ with $\bs((\{n_i\})_{i \in \omega}) \subseteq \X^c$. Observe that in $SF_N$, \II{} can always play in such a way that the sequence $(n_i)_{i \in \omega}$ is increasing. So in the first case, we get an increasing sequence $(n_i)_{i \in \omega}$ of elements of $N$ such that every block sequence of $(\{n_i\})_{i \in \omega}$ belongs to $\X$, or in other words, such that every infinite subset of $\{n_i \mid i \in \omega\}$ belongs to $\X$; and in the second case, in the same way, we get an infinite subset of $N$ every infinite subset of whose belongs to $\X^c$. 

\item For Gowers' theorem, let $E$ be a Banach space with a Schauder basis and work in the canonical approximate Gowers space $\G_E$ with the system $(\K_E, \oplus_E)$ of precompact sets introduced before. Given $Y \in P$, in $SF_Y$, whatever \I{} plays, \II{} can always ensure that the outcome will have the form $(S_{\R y_n})_{n \in \omega}$, where $(y_n)_{n \in \omega}$ is a block sequence. So given $\X \subseteq [E]$ analytic, $X \subseteq E$ a block subspace, and $\Delta$ a sequence of positive real numbers, Corollary \ref{AbstractGowers} gives us either a block sequence $(y_n)_{n \in \omega}$ in $X$ such that $\bs((S_{\R y_n})_{n \in \omega}) \subseteq \X^c$, or a block subspace $Y \subseteq X$ such that \II{} has a strategy in $G_Y$ to reach $(\X)_{\frac\Delta 2}$. In the first case, denoting by $Y$ the block subspace generated by the sequence $(y_n)$, this precisely means that $[Y] \subseteq \X^c$. In the second case, we have to be careful because the Gowers' game of the space $\G_E$ is not exactly the same as this defined in the introduction: in the one of the introduction, player \II{} is required to play vectors with finite support forming a block sequence, while in the one of $\G_E$, she can play any vector in the unit sphere of the subspace played by \I{}. This is not a real problem as, by perturbating the vectors given by her strategy, player \II{} can reach $\X_\Delta$ playing vectors with finite support; and without loss of generality, we can assume that the subspace $Y_n$ played by \I{} at the $(n+1)$\textsuperscript{th} turn is choosen small enough to force \II{} to play a $y_n$ such that $\supp(y_{n - 1}) < \supp(y_n)$.

\item To deduce Gowers' theorem for $c_0$, the method is the same except that this time, $\G_E$ satisfies the pigeonhole principle so Corollary \ref{AbstractGowers} will give us a conclusion with a strong asymptotic game in both sides.

\end{itemize}


To finish this section, let us show on an example that the hypothesis ``\I{} has a strategy in $F_p$ to reach $\X$'' does not always imply that for some subspace $q$, every sequence below $q$ satisfying some natural structural condition (for instance, being block) is in $\X_\Delta$. To see this, consider the Rosendal space $\Ros_K = (P, E \setminus \{0\}, \subseteq, \subseteq^*, \in)$ over a field $K$. We have the following fact:

\begin{fact}\label{Fact66}

Suppose that $K$ is a finite field. Let $\X \subseteq (E \setminus \{0\})^\omega$ and $X \in P$, and suppose that \I{} has a strategy in $F_X$ to reach $\X$. Then there exists a block subspace $Y \subseteq X$ such that every block sequence of $Y$ is in $\X$.

\end{fact}

\begin{proof}

Let $\K$ be the set of all sets of the form $F \setminus \{0\}$, where $F$ is a finite-dimensional subspace of $E$. Since the field $K$ is finite, the elements of $\K$ are finite too. For $F, G \subseteq E$ finite-dimensional, let $(F \setminus \{0\}) \oplus (G \setminus \{0\}) = (F + G) \setminus \{0\}$. Then $(\K, \oplus)$ is a system of precompact sets. The conclusion follows from Theorem \ref{AbstractApproxExact} applied to this system, using the same method as previously.


\end{proof}

Observe that this proof does not work when $K$ is infinite, and actually, this result is false. Let us give a counterexample. Let $(e_i)_{i < \omega}$ be the basis of $E$ with respect to whose block subspaces are taken, and let $\varphi : K^* \rightarrow \omega$ be a bijection. For $x \in E \setminus \{0\}$, let $N(x)$ be the first nonzero coordinate of $x$. Let $\mathcal{Y} = \{(x, y) \in (E\setminus\{0\})^2 \mid \varphi(N(x)) < \min \supp(y)\}$ and $\X = \{(x_n)_{n \in \omega} \in (E\setminus\{0\})^\omega \mid (x_0, x_1) \in \mathcal{Y}\}$. Then player \I{} has a strategy in $F_E$ to reach $\X$; this strategy is illustrated on the following diagram:

\smallskip

\begin{tabular}{ccccccc}
\textbf{I} & $E$ & & $\spa(\{e_i \mid i > \varphi(N(x))\})$ & \\
\textbf{II} & & $x$ & & $y$ 
\end{tabular}

\smallskip

\noindent But there is no block subspace $Y$ of $E$ such that every block sequence in $Y$ belongs to $\X$. Indeed, given $Y \subseteq E$ a block subspace generated by a block sequence $(y_n)_{n \in \omega}$, we can take $\lambda \in K$ such that $\varphi(N(\lambda y_0)) = \min \supp(y_1)$, and we have $(\lambda y_0, y_1, y_2, \ldots) \notin \X$. Note that, just like the counterexample to the pigeonhole principle presented in section \ref{par3}, this kind of counterexamples cannot be avoided by working in the projective Rosendal space:
in this space, a similar construction works taking for $N(x)$ the quotient of the last nonzero coordinate of $x$ by its first nonzero coordinate.

\smallskip

Therefore, in lots of cases, the ``subspaces'' of the form $(K_n)_{n \in \omega}$, where the $K_n$'s are elements of a system of precompact sets, cannot always be identified with ``genuine'' subspaces (i.e. elements of $P$): we always need a form of compactness for that.

\bigskip
\footnotesize
\noindent\textit{Acknowledgments.}
I would like to thank Stevo Todor\v{c}evi\'c who proposed me to work on the adversarial Ramsey principle and made me know Kastanas' game, thereby giving me all the tools to prove it. I would also like to thank the Isaac Newton Institute for Mathematical Sciences, Cambridge, for support and hospitality during the programme ``Mathematical, Foundational and Computational Aspects of the Higher Infinite'' where work on this paper was undertaken. This work was supported by EPSRC grant numbers EP/K032208/1 and EP/R014604/1. In particular, I met there Jordi L\'opez-Abad who suggested me to give an approximate version of my result in a metric setting, what I do in section \ref{par4}; I am grateful to him for that. I also thank Christian Rosendal who suggested me research directions in order to apply this result, and Valentin Ferenczi who gave me a very nice example of approximate Gowers space, that will be studied in a forthcoming paper with him and Wilson Cuellar Carrera. Finally, I would like to thank the anonymous referee for his useful comments that helped me much to improve the presentation of my results.

    \par
  \bigskip
    \textsc{\footnotesize Institut de Math\'ematiques de Jussieu-Paris Rive Gauche, Universit\'e Paris Diderot, Bo\^ite Courrier 7012, 75205 Paris Cedex 13, FRANCE}
    
    \textit{Current adress:} Universit\"at Wien, Institut f\"ur Mathematik, Kurt G\"odel Research Center, Augasse 2-6, UZA 1 -- Building 2, 1090 Wien, AUSTRIA
    
    \textit{Email address:} \texttt{noe.de.rancourt@univie.ac.at}

\end{document}